\documentclass[12pt]{amsart}
\usepackage{amsmath,amsthm,amsfonts,amssymb,times,latexsym,mathabx,url}
\usepackage[pdftex]{graphicx}
\usepackage{ bbold }

\graphicspath{}
\DeclareGraphicsExtensions{.pdf,.png,.jpg}
\newtheorem{theorem}{Theorem}[section]

\newtheorem{lem}[theorem]{Lemma}

\numberwithin{equation}{section}

\makeatletter
\renewcommand{\pmod}[1]{\allowbreak\mkern7mu({\operator@font mod}\,\,#1)}
\makeatother

\numberwithin{equation}{section}

\begin{document} 
	
	\title{Generalization of the Ford~--~Zaharescu Theorem}
	\author{Elizaveta D. Iudelevich and
		Vitalii V. Iudelevich}
	
	%\address{Moscow State University, Leninskie Gory str., 1, Moscow, Russia, 119991}
	%\address{Steklov Mathematical Institute,
	%	Gubkina str., 8, Moscow, Russia, 119991}
	%\email{vitaliiyudelevich@mail.ru}
	
	%\email{konyagin@mi-ras.ru}

	\begin{abstract}
	We derive an asymptotic formula for the sum
	$$
	H = \sum_{0<\gamma_k\leqslant T,\, 1\leqslant k\leqslant m}h(a_1\gamma_1+a_2\gamma_2+\cdots + a_m\gamma_m),
	$$
	where $a_1, a_2, \ldots, a_m$ are integers whose sum equals zero, $\gamma_1, \ldots, \gamma_m$ independently run through the imaginary parts of the non-trivial zeros of the Riemann zeta function, each zero occuring in the sum the number of times of its multiplicity, and the function $h$ belongs to some special class.
	\end{abstract}
	
	\date{}
	\maketitle
	
			\section{Introduction}
			Let us define the class of functions $\mathcal{H}_a \subseteq L^1(\mathbb{R})$ (for $a \geqslant 1$) that satisfy the following conditions:
			$$\int\limits_{-\infty}^{+\infty}h(x) dx =  0,\ \ \int\limits_{-\infty}^{+\infty} |x h(x)| dx< +\infty$$
			and $$|\hat{h}'(\xi)|\ll \frac{1}{(|\xi|+1)^{a+1}}$$ for each $\xi\in \mathbb{R}$. %такого что $|\xi|\geqslant 1.$
			
			Here, 
			$$
			\hat{h}(\xi) = \int\limits_{-\infty}^{+\infty} h(x)e^{-2\pi i x \xi} \, dx
			$$
			denotes the Fourier transform of the function $h$. 
			Let $\Lambda(n)$ denote the Mangoldt function,
			$$\Lambda(n) = 
			\begin{cases}
				\log p,\ & \text{if $n = p^m$, where $p$ is prime, and $m\geqslant 1$,}\\
				0, & \text{otherwise.}
			\end{cases}
			$$
			
			In 2015, K. Ford and A. Zaharescu proved the following statement (see \cite{Ford-Zahar2015}).
			\begin{theorem}\label{Thm1} Let $h \in \mathcal{H}_4$, and assume that the point $x = 0$ does not belong to the support of the function $h$. Then, as $T \to +\infty$, we have
				$$\sum_{0<\gamma, \gamma'\leqslant T}h(\gamma-\gamma') = \dfrac{T}{4\pi^2}\int\limits_{-\infty}^{+\infty}h(t)(\mathcal{K}_2(1+it)+\mathcal{K}_2(1-it))dt + O\left( \dfrac{T}{(\log T)^{{1}/{3}}}\right)\!,$$
				where the function $\mathcal{K}_2(s)$ is defined by the series 
				$$\mathcal{K}_2(s) = \sum_{n=1}^{+\infty}\dfrac{\Lambda^2(n)}{n^s}\ \ (\Re s> 1),$$
				and by analytic continuation to the region $\Re s = 1, s \neq 1$; $\gamma$ and $\gamma'$ denotes the imaginary parts of the non-trivial zeros of the Riemann zeta function, taking multiplicity into account, and the implied constant depends only on the function $h$.
			\end{theorem}
			
			%___
			
			In \cite{Iud2023}, the first author obtained an asymptotic formula for the sum
			$$
			H = \sum_{0<\gamma_k\leqslant T,\, 1\leqslant k\leqslant 4} h(\gamma_1+\gamma_2-\gamma_3-\gamma_4).
			$$
			Specifically, the following theorem was proven.
			
			\begin{theorem}
				Let $h \in \mathcal{H}_6,$
				%$$	H = H(T;h) =  \sum_{0<\gamma_k\leqslant T,\, 1\leqslant k\leqslant 4}h(\gamma_1+\gamma_2-\gamma_3-\gamma_4),$$
				and let the function $\mathcal{K}_4(s)$ be defined by the series
				\begin{equation*}
					\mathcal{K}_4(s) = \sum_{n=1}^{+\infty} \frac{\Lambda^4(n)}{n^s} \quad (\Re s > 1).
				\end{equation*}
				Then as $T \to +\infty,$ we have
				$$
				H = \frac{T^3}{24\pi^4} \int\limits_{-\infty}^{+\infty} h(t) \left(\mathcal{K}_4(2+it)+\mathcal{K}_4(2-it)\right) dt + O\left(\frac{T^3}{(\log T)^{1/4}}\right)\!,
				$$
				where the implied constant depends only on the function $h.$
			\end{theorem}
			
			This paper is devoted to generalizing these results to arbitrary linear combinations of ordinates of zeros. We prove the following
			
			\begin{theorem}\label{Th1}
				Let $m \geqslant 3,$ and let $\textbf{a} = (a_1, a_2, \ldots, a_m)$ be a fixed tuple of integers (not necessarily distinct) such that any two distinct numbers in the tuple are coprime, and 
				$$a_1+a_2+\cdots + a_m = 0.$$
				For $h \in \mathcal{H}_{m+2},$ define
				$$
				H = H(T; h, \textbf{a}) = \sum_{\substack{1\leqslant k\leqslant m \\ 0<\gamma_k\leqslant T}} h(a_1\gamma_1+a_2\gamma_2+\cdots+a_m\gamma_m),
				$$
				where $\gamma_1, \gamma_2,\ldots, \gamma_m,$ independently of each other, run through the ordinates of non-trivial zeros of the Riemann zeta function, taking multiplicity into account.
				
				Let the function $\mathcal{K}_m(s)$ be defined by the series
				$$
				\mathcal{K}_m(s) = \sum_{n=1}^{+\infty} \frac{\Lambda^m(n)}{n^s} \quad (\Re s > 1).
				$$
				
				Then as $T \to +\infty,$ the following asymptotic formula holds:
				$$
				H = D(m) T^{m-1} \int\limits_{-\infty}^{+\infty} h(t) (\mathcal{K}_m(S+it)+ \mathcal{K}_m(S-it)) dt + O\left(\frac{T^{m-1}}{(\log T)^{\frac{1}{m}}}\right)\!,
				$$
				where $S$ is the sum of the positive elements in $\textbf{a},$
				$$
				D(m) = \frac{(-1)^m C(m)}{(2\pi)^m},
				$$
				and
				$$
				C(m) = C(m; \textbf{a}) = \frac{1}{\pi} \int\limits_{-\infty}^{+\infty} \prod_{k=1}^m \left(\frac{\sin(|a_k|w)}{|a_k|w}\right) dw.
				$$
				
				The implied constant depends only on the function $h,$ $m,$ and $\textbf{a}.$
			\end{theorem}
			
			The simplest tuple satisfying Theorem \ref{Th1} has the form
			$$a_k = \begin{cases}
				+1,\ \  \text{if}\ 1\leqslant k \leqslant r, \\
				-1,\ \ \text{if}\  r< k \leqslant m,
			\end{cases}$$
			where $m = 2r$ and $r \geqslant 2.$ From Theorem \ref{Th1} and Lemma \ref{L4} of this paper, for $h \in \mathcal{H}_{2r+2},$ it follows that
			\begin{equation*}
				H (T; h, \textbf{\textit{a}})=
				c(r) T^{2r-1} \int\limits_{-\infty}^{+\infty}h(t)(\mathcal{K}_{2r}(r+it)+\mathcal{K}_{2r}(r-it))dt + O\left(\dfrac{T^{2r-1}}{(\log T)^{\frac{1}{2r}}} \right),
			\end{equation*}
			where 
			$$c(r) = \frac{r}{(2\pi)^{2r}}\sum_{k=1}^r k^{2r-3}\prod\limits_{\substack{1\leqslant \ell\leqslant r \\ \ell\neq k}}\dfrac{1}{k^2-\ell^2},$$
			and the implied constant depends only on $h$ and $r$.
			Below, we provide the values of $c(r)$ for small $r.$
			\begin{table}[h!]
				\centering
				\begin{tabular}{|c|c|c|}
					\hline
					$r$ & $c(r)\pi^{2r}$  & value \\
					\hline
					1 & ${1}/{4}$ & $0{.}25$ \\
					2 & ${1}/{24}$ & $0{.}0416666\ldots$ \\
					3 & ${11}/{1\,280}$ & $0.0085937\ldots$ \\
					4 & ${151}/{80\,640}$ & $0{.}0018725\ldots$ \\
					5 & ${15\,619}/{37\,158\,912}$ & $0{.}0004203\ldots$ \\
					\hline
				\end{tabular}
				%	\caption{Значения коэффициента $c(r)$ для $r=1,\ldots,5$}
				%\label{tab:cr_values}
			\end{table}
			
			Interest in the sums above is explained by an effect discovered by R. P. Marco \cite{Marco-Perez2011}, which has been termed the \textit{repulse phenomenon}. When forming pairwise differences of the ordinates of the zeros of the Riemann zeta function, it can be observed that these differences very rarely take values close to the ordinates of the zeros of the Riemann zeta function. In other words, zeros with large ordinates "know" about zeros with small ordinates and, in a certain sense, "avoid" them.
			
			An illustration of this effect can be seen in the graph of the function $\mathcal{K}_2(1+it)+\mathcal{K}_2(1-it)$ (see \cite{Ford-Zahar2015}), which exhibits "dips" at the ordinates of the zeros of the Riemann zeta function. Similarly, graphs of functions $y(t) = y(t; \textbf{\textit{a}}) = \mathcal{K}_m(S+it)+\mathcal{K}_m(S-it)$ also show dips at the ordinates of the zeros of the Riemann zeta function, particularly noticeable for small values of $S$. Below are graphs for 
			\begin{equation}\label{aaa}
				\textbf{\textit{a}} \in \left\{(1, 1, -2),\ (1, 1, -1, -1),\ (1, 2, -3)\right\},
			\end{equation}
			depicted in red, blue, and green colors respectively.
			
			\begin{figure}[h!]
				\caption{The graphs of the functions $y(t; \textbf{a})$ and the local minima at the points $\gamma_1 \approx 14.134, \gamma_2 \approx 21.022, \gamma_3 \approx 25.01, \gamma_4 \approx 30.424, \gamma_5 \approx 32.935, \gamma_6 \approx 37.586$.}
				\centering{\includegraphics[scale=0.9]{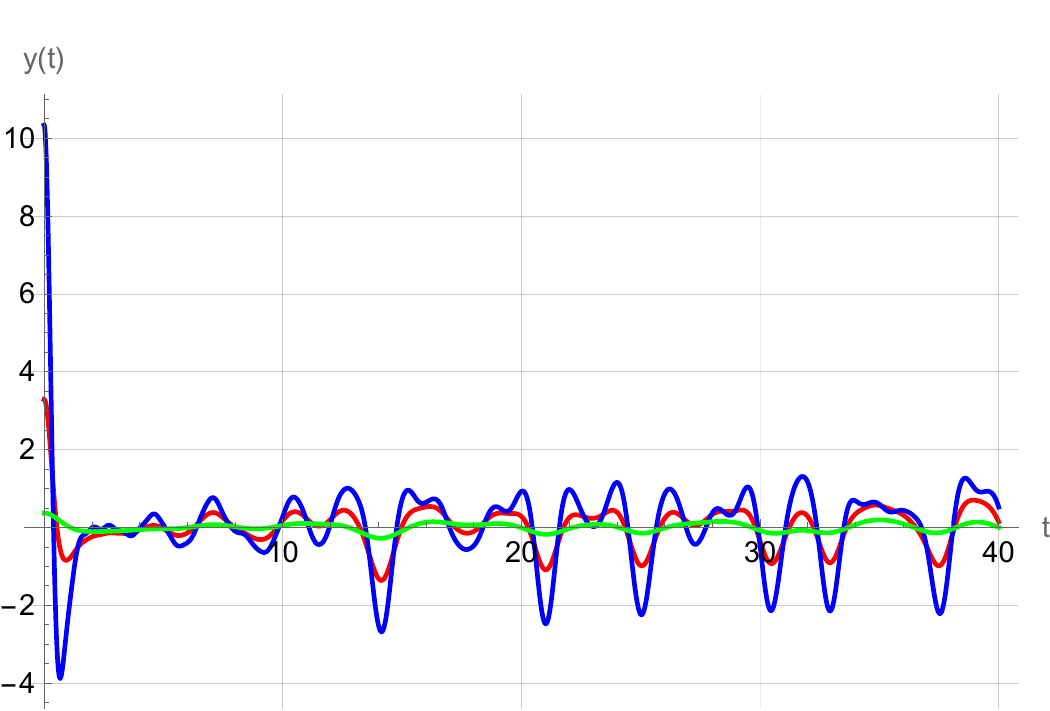}}
				\label{fig:image1}
			\end{figure}
			
			Following \cite{Ford-Zahar2015}, the presence of dips can be explained as follows. Let 
			$$
			b_m(k) = \sum_{\delta|k} \mu(\delta)\delta^{m-1}.
			$$
			Then, using the identity 
			$$
			\sum_{d\delta = k} d^{m-1} b_m(\delta) = 1,
			$$
			we find
			$$
			\mathcal{K}_m(s) = \sum_{\delta = 1}^{+\infty} b_m(\delta) G_m(\delta s),
			$$
			where
			$$
			G_m(s) = \sum_{n = 1}^{+\infty} \frac{\Lambda(n)(\log n)^{m-1}}{n^s} = (-1)^m \left(\frac{\zeta'(s)}{\zeta(s)}\right)^{(m-1)}.
			$$
			
			Using the estimates $|b_m(\delta)| \leqslant \delta^{m-1}$ and $G_m(\sigma+it) \ll_m 2^{-\sigma}$ for $\sigma > 1,$ as well as the identity (see \cite[chapter~I, \S 6, Theorem~2]{Voronin-Karatsuba})
			$$
			\frac{\zeta'(s)}{\zeta(s)} = -\frac{1}{s-1} + \sum_{\rho} \left(\frac{1}{s-\rho} + \frac{1}{\rho}\right) + c,
			$$
			where $c$ is a constant and $\rho$ runs through all zeros of the Riemann zeta function (including trivial ones), we obtain for $M \geqslant 2:$
			$$\mathcal{K}_m(s) = (m-1)!\sum_{\delta\leqslant M}\dfrac{b_m(\delta)}{\delta^m} \left(\dfrac{1}{(s-\frac{1}{\delta})^m}-\sum_{\rho}\dfrac{1}{(s-\frac{\rho}{\delta})^m}\right) + O_m\left(M^{m-1}2^{-M\sigma}\right).$$
			If $s = S + it$, the largest (in absolute value) negative term in the sum above occurs when $m = 1$, $t \approx \Im \rho_0$, and $\rho = \rho_0$, where $\rho_0$ is a non-trivial zero of the Riemann zeta function. In this case, the term is approximately equal to 
			$$
			-\frac{(m-1)!}{(S-\frac{1}{2})^m}.
			$$
			
			Thus, by ignoring contributions for $m \geqslant 2$ and $\rho \neq \rho_0$, we obtain
			$$
			y(t; \textbf{\textit{a}}) \approx -\frac{2(m-1)!}{(S-\frac{1}{2})^m}
			$$
			for $t \approx \Im \rho_0.$ Consequently, near the ordinates of the zeros of the Riemann zeta function, the function $y(t; \textbf{\textit{a}})$ will have local minima and take values close to 
			$$
			-2\frac{(m-1)!}{(S-\frac{1}{2})^m}.
			$$
			In particular, for $\textbf{\textit{a}}$ from \eqref{aaa}, these values are approximately equal to $-1.18, -2.37, -0.26,$ as can be observed in the graph above.
			
			In this paper, we follow the techniques from \cite{Ford-Zahar2015}, while also introducing additional arguments such as Lemmas~\ref{L1} and~\ref{L2}, as well as considerations related to simultaneous approximation of quantities $x$ and $x^d$ by integers for real $x$ and integer $d \geqslant 2.$
			
			The following notations are adopted in this work.
			
			The notation $f(x) \ll g(x)$ for $g(x) > 0$ means that there exists a constant $c > 0$ such that $|f(x)| \leqslant c g(x).$
			
			$N(T)$ denotes the number of zeros (counting multiplicities) of the Riemann zeta function in the region $0 < \Re s < 1,\ 0 < \Im s \leqslant T.$ The following equality holds (see \cite[chapter~1, \S8]{Voronin-Karatsuba}):
			\begin{equation*}
				N(T) = \frac{T}{2\pi} \log{\frac{T}{2\pi}} - \frac{T}{2\pi} + O(\log T).
			\end{equation*}
			
			$N(\sigma,T)$ denotes the number of zeros of the Riemann zeta function in the region 
			$$
			\sigma \leqslant \Re s \leqslant 1,\quad 0 < \Im s \leqslant T.
			$$
			As is known (see~\cite{Selberg1946}), uniformly for $\sigma \geqslant 1/2,\ T \geqslant 2,$ the following estimate 
			\begin{equation*}
				N(\sigma,T) \ll T^{1-\frac{1}{4}(\sigma-\frac{1}{2})} \log T
			\end{equation*}
			holds.
			For $x \in \mathbb{R}$, the symbol $n(x)$ denotes the nearest integer to $x.$ Specifically, we write $n(x)= n,$ if $x \in (n-\frac{1}{2}, n+\frac{1}{2}),$ and $n(x)= n+1,$ if $x = n+\frac{1}{2},$ where $n \in {\mathbb Z}.$

			%Формулировка теоремы:
			
			%Пусть $\textbf{a} = (a_1, a_2, \ldots, a_m) \in (\mathbb{Z}^x)^m$, где $m \geqslant 3$, $a_1 + \ldots + a_m = 0$ и $a_i \ne a_j$, $(a_i, a_j) = 1$. 
			
			%Пусть $S = \sum\limits_{\substack{k=1 \\ a_k > 0}}^{m} a_k$. Тогда для любой функции $h \in \mathcal{H}_{m+2}$ справедлива асимптотическая формула 
			%\[
			%\sum\limits_{\overrightarrow{\gamma}} h(\overrightarrow{a} \cdot \overrightarrow{\gamma}) = D(m) T^{m-1} \int\limits_{-\infty}^{+\infty} h(t) (\mathcal{K}_m(S+it)+ \mathcal{K}_m(S-it) ) dt + O_{\overrightarrow{a}, h} \left(\dfrac{T^{m-1}}{\log T} \right),
			%\]
			%где $\overrightarrow{\gamma} = (\gamma_1, \gamma_2, \ldots, \gamma_m)$ пробегает векторы их нетривиальных нулей дзета-функции Римана с условием $0 < \gamma_k \leqslant T$, $ 1\leqslant k \leqslant m$,
			%\[
			%D(m) = \dfrac{(-1)^m C(m)}{(2\pi)^{m+1}},
			%\]
			%\[
			%C(n) = \dfrac{2}{\prod\limits_{k=1}^m |a_k|} \int\limits_{-\infty}^{+\infty} \prod \limits_{k=1}^m \left( \dfrac{\sin (|a_k|w)}{w}\right) dw
			%\]
			%и \[
			%K_m(s) = \prod\limits_{n=1}^{+\infty} \dfrac{\Lambda^m (n)}{n^s}.
			%\]
			
			\section{Auxiliary Results}
			We will need the following lemmas, which are of independent interest.
			\begin{lem}\label{L1}
				Let $x_1, x_2, \ldots, x_q \in \mathbb{C}$ and $1 \leqslant r < q$. Then we have
				\begin{equation}\label{ident}
					\sum\limits_{s=1}^q (-1)^{s-1} \sum\limits_{1 \leqslant k_1 < \cdots < k_s \leqslant q} \sum\limits_{\substack{j_{k_1}, \ldots,j_{k_s} \geqslant 0: \\ j_{k_1} + \cdots + j_{k_s}=r}} \binom{2r}{2j_{k_1}, \ldots,  2j_{k_s}} x_{k_1}^{j_{k_1}} \cdots x_{k_s}^{j_{k_s}} = 0, 
				\end{equation}
				where
				$$\binom{2r}{2j_{k_1}, \ldots,  2j_{k_s}}  = \dfrac{(2r)!}{(2j_{k_1})!  \cdots (2j_{k_s})!}.$$
			\end{lem}
			
			\begin{proof}
				The left-hand side of the equality \eqref{ident} is a linear combination of monomials of the form
				\begin{equation}\label{monom}
					x_{i_1}^{t_{1}} x_{i_2}^{t_{2}} \cdots x_{i_p}^{t_{p}},
				\end{equation}
				where $t_1, t_2, \ldots, t_p \geqslant 1$ and $t_1 + \cdots + t_p = r$. From here, we obtain $p \leqslant t_1 + \cdots + t_p = r < q$. We will show that the coefficient of the monomial \eqref{monom} in \eqref{ident} is equal to zero. Indeed, the sought coefficient is given by
				\begin{multline*}
					\sum\limits_{s=p}^q (-1)^{s-1}  \sum\limits_{\substack{1 \leqslant k_1 < \cdots < k_s \leqslant q \\ \{i_1,\ldots, i_p\}\subseteq \{k_1,\ldots, k_s\}}} \ \ \sum\limits_{\substack{j_{k_\ell}=0\ \text{for}\ k_\ell \not \in \{i_1, \ldots, i_p\} \\ j_{k_\ell}=t_m\ \text{for}\ k_\ell=i_m,\,1\leqslant m\leqslant p}} \binom{2r}{2j_{k_1}, \ldots, 2j_{k_s}} =  \\
					(-1)^{p-1}  \binom{2r}{2t_1, \ldots, 2t_p}		\sum\limits_{s=p}^q (-1)^{s-p}  \sum\limits_{\substack{1 \leqslant k_1 < \cdots < k_s \leqslant q \\ \{i_1,\ldots, i_p\}\subseteq \{k_1,\ldots, k_s\}}} 1 = \\    (-1)^{p-1}  \binom{2r}{2t_1, \ldots, 2t_p}		\sum\limits_{s=p}^q (-1)^{s-p} \binom{q-p}{s-p} = \\  (-1)^{p-1}  \binom{2r}{2t_1, \ldots, 2t_p}	(1-1)^{q-p} = 0.
				\end{multline*}
				The result follows.
			\end{proof}
			
			\begin{lem}\label{L2}
				Let $q \geqslant 2$ and $\alpha_1, \alpha_2, \ldots, \alpha_q$ be arbitrary complex numbers. Then for $1 \leqslant r < q$, the following equality holds:
				\begin{equation}\label{ident2}
					\sum\limits_{s=1}^q 2^{q-s} (-1)^{s-1} \sum\limits_{1 \leqslant k_1 < \cdots < k_s \leqslant q} \ \  \sum\limits_{\substack{\varepsilon_i \in \{-1, +1\} \\ 2 \leqslant i \leqslant s}} \left(\alpha_{k_1}+\varepsilon_2\alpha_{k_2} + \cdots + \varepsilon_s\alpha_{k_s} \right)^{2r}=0. 
				\end{equation}
			\end{lem}
			\begin{proof}
				Let $W$ denote the left-hand side of the equality \eqref{ident2}. Then
				\[
				W = \dfrac{1}{2} \sum\limits_{s=1}^q 2^{q-s} (-1)^{s-1} W_s,
				\]
				where
				\[
				W_s = \sum\limits_{1 \leqslant k_1 < \cdots < k_s \leqslant q} \ \  \sum\limits_{\substack{\varepsilon_i \in \{-1, +1\} \\ 1 \leqslant i \leqslant s}} \left(\varepsilon_1\alpha_{k_1}+\varepsilon_2\alpha_{k_2} + \cdots + \varepsilon_s\alpha_{k_s} \right)^{2r}.
				\]
				Expanding the brackets, we obtain
				\begin{multline*}
					W_s = \sum\limits_{1 \leqslant k_1 < \cdots < k_s \leqslant q} \ \  \sum\limits_{\substack{\varepsilon_i \in \{-1, +1\} \\ 1 \leqslant i \leqslant s}} \ \  \sum\limits_{\substack{\nu_j \geqslant 0 \\ \nu_1 + \cdots + \nu_s=2r }} \binom{2r}{\nu_1, \ldots, \nu_s} (\varepsilon_1 \alpha_{k_1})^{\nu_1} \cdots  (\varepsilon_s \alpha_{k_s})^{\nu_s} = \\ \sum\limits_{\substack{\nu_j \geqslant 0 \\ \nu_1 + \cdots + \nu_s=2r}} \binom{2r}{\nu_1, \ldots, \nu_s}  \sum\limits_{1 \leqslant k_1 < \cdots < k_s \leqslant q} \alpha_{k_1}^{\nu_1} \cdots \alpha_{k_s}^{\nu_s} V_s,
				\end{multline*}
				where \[
				\binom{2r}{\nu_1, \ldots, \nu_s} = \dfrac{2r!}{\nu_1! \cdots \nu_s!}
				\]
				and
				\[
				V_s = V_s(\nu_1, \ldots, \nu_s) = \sum\limits_{\substack{\varepsilon_i \in \{-1, +1\} \\ 1 \leqslant i \leqslant s}} \varepsilon_1^{\nu_1} \cdots \varepsilon_s^{\nu_s}.
				\]
				Note that if at least one of the numbers $\nu_1, \ldots, \nu_s$ is odd, then $V_s = 0$. Indeed, let us assume that $\nu_1$ is odd. Then
				\[
				V_s = \sum\limits_{\substack{\varepsilon_i \in \{-1, +1\} \\ 1 \leqslant i \leqslant s}} \varepsilon_1 \varepsilon_2^{\nu_2} \cdots \varepsilon_s^{\nu_s} =  \sum\limits_{\substack{\varepsilon_i \in \{-1, +1\} \\ 2 \leqslant i \leqslant s}} \varepsilon_2^{\nu_2} \cdots \varepsilon_s^{\nu_s} - \sum\limits_{\substack{\varepsilon_i \in \{-1, +1\} \\ 2 \leqslant i \leqslant s}} \varepsilon_2^{\nu_2} \cdots \varepsilon_s^{\nu_s} = 0.
				\]
				On the other hand, if all the numbers $\nu_1, \ldots, \nu_s$ are even, then obviously $V_s = 2^s$. Thus, by setting $\nu_\ell = 2j_\ell$ for $1 \leqslant \ell \leqslant s$, we find:
				\[
				W_s = 2^s \sum\limits_{\substack{ j_\ell \geqslant 0 \\ j_1 + \cdots + j_s=r}} \binom{2r}{2j_1, \ldots, 2j_s}  \sum\limits_{1 \leqslant k_1 < \cdots < k_s \leqslant q} \alpha_{k_1}^{2j_1} \cdots  \alpha_{k_s}^{2j_s},
				\]
				and by virtue of Lemma \ref{L1}, for the sum $W$ we obtain
				\[
				W = 2^{q-1} \sum\limits_{s=1}^q (-1)^{s-1} \sum\limits_{\substack{j_\ell \geqslant 0 \\ j_1 + \cdots + j_s=r }} \binom{2r}{2j_1, \ldots, 2j_s}  \sum\limits_{1 \leqslant k_1 < \cdots < k_s \leqslant q}  (\alpha_{k_1}^2)^{j_1} \cdots  (\alpha_{k_s}^2)^{j_s} = 0.
				\]
				This concludes the proof.
			\end{proof}
			
			\begin{lem}\label{L3}
				Let $A_1, A_2, \ldots, A_s$ ($s \geqslant 2$) be a certain sequence of real numbers. Then 
				\[
				\prod\limits_{\ell=1}^s 2 \cosh (A_\ell) =  \sum\limits_{\substack{\varepsilon_j \in \{-1, +1\} \\ 2 \leqslant j \leqslant s}} 2\cosh (A_1 + \varepsilon_2A_2 + \cdots +\varepsilon_sA_s).
				\]
			\end{lem}
			\begin{proof}
				We have
				\begin{multline*}
					P = \prod\limits_{\ell=1}^s \left( e^{A_\ell} + e^{-A_\ell}\right) = \sum\limits_{\substack{\varepsilon_j \in \{-1, +1\} \\ 1 \leqslant j \leqslant s}} e^{\varepsilon_1A_1 + \varepsilon_2A_2 + \cdots + \varepsilon_sA_s} = \\
					\sum\limits_{\substack{\varepsilon_j \in \{-1, +1\} \\ 2 \leqslant j \leqslant s}} e^{A_1 + \varepsilon_2A_2 + \cdots + \varepsilon_sA_s} + \sum\limits_{\substack{\varepsilon_j \in \{-1, +1\} \\ 2 \leqslant j \leqslant s}} e^{-\left(A_1 - \varepsilon_2A_2 - \cdots - \varepsilon_sA_s\right)}.
				\end{multline*}
				Since $\varepsilon_j$ and $-\varepsilon_j$ simultaneously run through the set $\{-1, +1\}$, we have
				\begin{multline*}
					P= \sum\limits_{\substack{\varepsilon_j \in \{-1, +1\} \\ 2 \leqslant j \leqslant s}} e^{A_1 + \varepsilon_2A_2 + \cdots + \varepsilon_sA_s} + \sum\limits_{\substack{\varepsilon_j \in \{-1, +1\} \\ 2 \leqslant j \leqslant s}} e^{-\left(A_1 + \varepsilon_2A_2 + \cdots + \varepsilon_sA_s\right)} = \\   \sum\limits_{ \substack{\varepsilon_j \in \{-1, +1\} \\ 2\leqslant j\leqslant s }} 2 \cosh (A_1 + \varepsilon_2A_2 + \cdots +\varepsilon_sA_s).
				\end{multline*}
				The lemma is proved.
			\end{proof}

			\begin{lem}\label{L4}
				For $n \geqslant 1$, we have
				\[
				\int\limits_{-\infty}^{+\infty} \left( \dfrac{\sin t}{t}\right)^{2n} dt = \pi n \sum\limits_{k=1}^n k^{2n-3} \prod\limits_{\substack{1 \leqslant \ell \leqslant n \\ \ell \ne k}} \dfrac{1}{k^2-\ell^2}.
				\]
			\end{lem}
			
			\begin{proof}
				Let $\mathcal{L}\{f\}$ and $\mathcal{L}^{-1}\{g\}$ be the direct and inverse Laplace transforms of the functions $f$ and $g$, respectively. Then, using the well-known identity
				\[
				\int\limits_{0}^{+\infty} f(t)g(t)dt = 	\int\limits_{0}^{+\infty}\mathcal{L}\{f\}(s)\mathcal{L}^{-1}\{g\}(s) ds
				\]
				and the equalities \cite[chapter IV, \S 4.7, (3)]{Erd-Bate} and \cite[chapter V, \S 5.4, (1)]{Erd-Bate}
				\[
				\mathcal{L}\{\sin^{2n}t\} = \dfrac{(2n)!}{s(s^2+2^2)(s^2+4^2) \cdots (s^2 + (2n)^2)},
				\]
				\[
				\mathcal{L}^{-1}\{t^{-2n}\} = \dfrac{s^{2n-1}}{(2n-1)!},
				\]
				we get
				\[
				\int\limits_{0}^{+\infty} \left( \dfrac{\sin t}{t}\right)^{2n} dt = 2n 	\int\limits_{0}^{+\infty} \dfrac{s^{2n-2}}{(s^2+2^2)(s^2+4^2) \cdots (s^2 + (2n)^2)}ds.
				\]
				By evaluating the resulting improper integral using residues, we obtain
				\begin{multline*}
					\int\limits_{-\infty}^{+\infty} \left( \dfrac{\sin t}{t}\right)^{2n} dt = 4 \pi i n \sum\limits_{k=1}^n \underset{s=2ki}{\,{\text{res}}}\, \dfrac{s^{2n-2}}{(s^2+2^2)(s^2+4^2) \cdots (s^2 + (2n)^2)} = \\
					\pi n \sum\limits_{k=1}^n \dfrac{2^{2n-2}k^{2n-2} (-1)^{n-1}}{k} \prod\limits_{\substack{1 \leqslant \ell \leqslant n \\ \ell \ne k}} \dfrac{1}{-4k^2+4\ell^2} =   \pi n  \sum\limits_{k=1}^n  k^{2n-3}  \prod\limits_{\substack{1 \leqslant \ell \leqslant n \\ \ell \ne k}} \dfrac{1}{k^2-\ell^2}.
				\end{multline*}
				The claim follows. 
			\end{proof}
			
			For brevity, let 
			$$
			\Delta = \Delta (\gamma_1, \gamma_2, \ldots, \gamma_m) = a_1\gamma_1 + a_2\gamma_2 + \cdots + a_m\gamma_m,
			$$
			where for each $1 \leqslant k \leqslant m$, the number $\gamma_k$ is the imaginary part of some non-trivial zero $\zeta(s)$, and $a_k$ is a non-zero integer as per the conditions of Theorem \ref{Th1}. Let us set 
			$$
			|a_1| + |a_2| + \cdots + |a_m| = 2r \quad (r \geqslant 1).
			$$
			We need the following lemma.% Будем считать, что набор $a_1, a_2, \ldots, a_m$~--- фиксированный. Не ограничивая общности будем считать, что $a_1 > 0$ и $a_m < 0$. 
			
			\begin{lem}\label{L5}
				Let $K \geqslant 1$, $T \geqslant 2$ and 
				\[
				W = \sum\limits_{\substack{0 < \gamma_k \leqslant T, \\ 1 \leqslant k \leqslant m}} \min \left(K, \dfrac{1}{|\Delta|}\right)\!.
				\]
				Then we have
				\[
				W \ll T^{m-1} (\log T)^m (K+ \log T),
				\]
				where the implied constant depends only on $a_1, a_2, \ldots, a_m.$
			\end{lem}
			
			\begin{proof}
				Without loss of generality, we can assume that $a_m > 0$. We have
				\[
				W = K \sum\limits_{\substack{0 < \gamma_k \leqslant T, \\ 1 \leqslant k \leqslant m, \\ |\Delta| \leqslant \frac{1}{K}}} 1 +  \sum\limits_{\substack{0 < \gamma_k \leqslant T, \\ 1 \leqslant k \leqslant m, \\ |\Delta| > \frac{1}{K}}} \dfrac{1}{|\Delta|} = W_1 + W_2.
				\]
				Let us estimate the sum $W_1$. We set $\Delta' = a_1\gamma_1 + \cdots + a_{m-1}\gamma_{m-1}$. Using the inequality $||a| - |b|| \leqslant |a - b|$, we will have
				\begin{multline*}
					W_1 \leqslant K \sum\limits_{\substack{0 < \gamma_i \leqslant T, \\ 1 \leqslant i \leqslant m, \\ |a_1\gamma_1 + \cdots +a_m \gamma_m| \leqslant 1}} 1 \leqslant K 	 \sum\limits_{\substack{0 < \gamma_k \leqslant T, \\ 1 \leqslant k \leqslant m, \\ |a_m\gamma_m- |\Delta'|| \leqslant 1}}1 = \\  K \sum\limits_{\substack{0 < \gamma_k \leqslant T, \\ 1 \leqslant k \leqslant m-1}} \sum\limits_{\substack{0 < \gamma_m \leqslant T \\ \gamma_m \leqslant \frac{1}{a_m} (|\Delta'| +1)\\ \gamma_m \geqslant \frac{1}{a_m} (|\Delta'| -1)}} 1 \ll\\
					K (\log T) (N(T))^{m-1} \ll T^{m-1} (\log T)^m K.
				\end{multline*}
				Now let us estimate the sum $W_2$. We have
				\[
				W_2 = \sum\limits_{\substack{0 < \gamma_k \leqslant T, \\ 1 \leqslant k \leqslant m, \\ \frac{1}{K} < |\Delta| \leqslant 1}} \dfrac{1}{|\Delta|} + \sum\limits_{\substack{\nu \geqslant 1 \\ \nu \leqslant 2rT}} \sum\limits_{\substack{0 < \gamma_k \leqslant T, \\ 1 \leqslant k \leqslant m, \\ \nu < |\Delta| \leqslant \nu+ 1}} \dfrac{1}{|\Delta|} = W_3+ W_4.
				\]
				According to the calculations presented above,
				\[
				W_3 \ll K T^{m-1} (\log T)^m.
				\]
				For the sum $W_4$, we have
				$$
				W_4 \ll \sum\limits_{\substack{\nu \geqslant 1 \\ \nu \leqslant 2rT}} \dfrac{1}{\nu} \sum\limits_{\substack{0 < \gamma_k \leqslant T, \\ 1 \leqslant k \leqslant m, \\ \nu < |\Delta| \leqslant \nu+ 1}} 1 =  \sum\limits_{\substack{\nu \geqslant 1 \\ \nu \leqslant 2rT}} \dfrac{1}{\nu} \sum\limits_{\substack{0 < \gamma_k \leqslant T, \\ 1 \leqslant k \leqslant m-1}} \sum\limits_{\substack{0 < \gamma_m \leqslant T \\  \nu < |\Delta| \leqslant \nu+ 1}} 1.
				$$
				Fixing $\gamma_1, \gamma_2, \ldots, \gamma_{m-1}$, we see that $\gamma_m$ varies over an interval of length no more than 1. Indeed, since
				$
				\nu < |\Delta' + a_m\gamma_m| \leqslant \nu + 1,
				$
				there are two possible cases:
				\begin{itemize}
					\item $\nu < \Delta'  + a_m\gamma_m \leqslant \nu + 1$, and then
					\[
					\dfrac{\nu-\Delta'}{a_m} < \gamma_m \leqslant \dfrac{\nu-\Delta'+1}{a_m};
					\]
					\item $-(\nu+1) \leqslant \Delta'+a_m\gamma_m < - \nu$, and then
					\[
					-\dfrac{\nu+\Delta'+1}{a_m}  \leqslant \gamma_m < -\dfrac{\nu+\Delta'}{a_m}.
					\]
				\end{itemize}
				Thus, since the number of such $\gamma_m$ does not exceed $O_m(\log T)$, we obtain
				\[
				W_4 \ll (\log T)  (N(T))^{m-1} \sum\limits_{\nu \leqslant 2rT} \dfrac{1}{\nu} \ll_m T^{m-1} (\log T)^{m+1}.
				\]
				The lemma is proved.
			\end{proof}
			
			\begin{lem}\label{L6}
				Let $U \geqslant 2$, $T \geqslant 2$, and $m \geqslant 1$ be fixed integers, and
				\[
				D(\xi) = \sum\limits_{0 < \gamma \leqslant T} e^{2\pi i \xi \gamma} \left(1 - e^{2\pi \xi \left( \beta - \frac{1}{2}\right) } \right) .
				\]
				Then, for $U \leqslant \frac{\log T}{32 \pi m}$, the following estimate 
				\[
				\int\limits_0^U |D(\xi)|^{2m} d \xi \ll_m T^{2m-1} \left(\dfrac{U^{4m}}{(\log T)^{2m-1}} + (\log T)^{2m+1} T^{-\frac{1}{4U}}\right)
				\]
				holds.
			\end{lem}
			
			\begin{proof} 
				Since the quantities $\rho = \beta+i\gamma$ and $1-\bar{\rho} = 1-\beta+i\gamma$ simultaneously run through the zeros of the Riemann zeta function, whose imaginary parts are positive and do not exceed $T$, we have
				\[
				D(\xi) = \dfrac{1}{2} \sum\limits_{0 < \gamma \leqslant T} e^{2\pi i \xi \gamma} \left(2 - e^{2\pi \xi \left( \beta - \frac{1}{2}\right) }-e^{-2\pi \xi \left( \beta - \frac{1}{2}\right)} \right).
				\]
				Hence, we get
				\[
				|D(\xi)|^{2m} = \dfrac{1}{4^m} \sum\limits_{\substack{0 < \gamma_k \leqslant T \\ 1 \leqslant k \leqslant 2m}}  e^{2\pi i \xi(\gamma_1 + \cdots + \gamma_m-\gamma_{m+1} - \cdots - \gamma_{2m}) }  \prod\limits_{k=1}^{2m} \left( 2-e^{2\pi \xi \left( \beta_k - \frac{1}{2}\right) }-e^{-2\pi \xi \left( \beta_k - \frac{1}{2}\right)}\right).
				\]
				Therefore, we will have
				\[
				J= \int\limits_0^U 	|D(\xi)|^{2m} d\xi =  \dfrac{1}{4^m} \sum\limits_{\substack{1 \leqslant k \leqslant 2m \\ 0 < \gamma_k \leqslant T}} \int\limits_0^U  I(\xi) d\xi,
				\]
				where
				\[
				I(\xi) = I(\rho_1, \rho_2, \ldots, \rho_{2m}, \xi) = e^{2\pi i \xi \Delta} \prod\limits_{k=1}^{2m}\left( 2-e^{2\pi \xi \left( \beta_k - \frac{1}{2}\right) }-e^{-2\pi \xi \left( \beta_k - \frac{1}{2}\right)}\right)
				\]
				and
				$$
				\Delta = \gamma_1 + \cdots +\gamma_m - \gamma_{m+1}-\cdots - \gamma_{2m}
				$$
				(note that in this and the previous lemmas, $\Delta$ denotes different quantities).
				Set
				$$\alpha = \max\limits_{1 \leqslant k \leqslant 2m} \left|\beta_k-\dfrac{1}{2}\right|$$
				and we split the integral $J$ into four terms depending on the values of $\alpha$ and $\Delta$. Namely,
				\[
				J = \dfrac{1}{4^m} \left( J_1 + J_2 + J_3 + J_4\right),
				\]
				where in $J_1$ are included exactly those $\gamma_k$ for which $|\Delta| \leqslant 1$ and $\alpha \leqslant U^{-1}$, in $J_2$ those $\gamma_k$ for which $|\Delta| \leqslant 1$ and $\alpha > U^{-1}$,
				in $J_3$ those $\gamma_k$ for which $|\Delta| > 1$ and $\alpha \leqslant U^{-1}$, and finally, in $J_4$ those $\gamma_k$ for which $|\Delta| > 1$ and $\alpha > U^{-1}$.	
				
				First, let us estimate the quantity $J_1$. Since $$\alpha = \max\limits_{1 \leqslant k \leqslant 2m} \left|\beta_k-\dfrac{1}{2}\right| \leqslant U^{-1},$$ it follows that
				\[
				\left|2\pi \xi \left(\beta_k-\dfrac{1}{2} \right)    \right| \leqslant 2\pi UU^{-1} = 2\pi.
				\]
				Since $\left|2 - e^x-e^{-x}  \right| \ll x^2$ for $|x| \leqslant 2\pi$, we have 
				\[
				\left|  2-e^{2\pi \xi \left( \beta_k - \frac{1}{2}\right) }-e^{-2\pi \xi \left( \beta_k - \frac{1}{2}\right)}  \right| \ll \xi^2  \left|\beta_k-\dfrac{1}{2}\right| ^2 \ll \alpha^2U^2.
				\]
				Therefore,
				\[
				\prod\limits_{k=1}^{2m} \left|  2-e^{2\pi \xi \left( \beta_k - \frac{1}{2}\right) }-e^{-2\pi \xi \left( \beta_k - \frac{1}{2}\right)}  \right| \ll \alpha^{4m} U^{4m}
				\]	
				and
				\[
				\int\limits_0^U I(\xi) d\xi \ll \alpha^{4m} U^{4m+1}.
				\]
				Thus, we find
				\begin{multline*}
					J_1 = \sum\limits_{\substack{1 \leqslant k \leqslant 2m \\ 0<\gamma_k \leqslant T \\ |\Delta| \leqslant 1,\, \alpha \leqslant U^{-1}}} \int\limits_0^U I(\xi) d\xi \ll  U^{4m+1} \sum\limits_{\substack{1 \leqslant k \leqslant 2m \\ 0<\gamma_k \leqslant T \\ |\Delta| \leqslant 1,\,\alpha \leqslant U^{-1}}}  \alpha^{4m} =\\
					4m U^{4m+1} \sum\limits_{\substack{1 \leqslant k \leqslant 2m \\ 0<\gamma_k \leqslant T \\ |\Delta| \leqslant 1,\,\alpha \leqslant U^{-1}}} \int\limits_0^\alpha \sigma^{4m-1} d\sigma.
				\end{multline*}
				Using the fact that the non-trivial zeros of the Riemann zeta function are symmetrically located with respect to the line $\Re s = {1}/{2}$, we have
				\begin{equation*}
					J_1\ll_m U^{4m+1} \int\limits_0^{U^{-1}} \sigma^{4m-1} \sum\limits_{\substack{1 \leqslant k \leqslant 2m \\ 0<\gamma_k \leqslant T \\ |\Delta| \leqslant 1,\, \alpha \geqslant \sigma }} 1 \ \  d\sigma \ll_m  U^{4m+1} \int\limits_0^{U^{-1}} \sigma^{4m-1} \sum\limits_{\substack{1 \leqslant k \leqslant 2m \\ 0<\gamma_k \leqslant T \\ |\Delta| \leqslant 1 \\  \max\limits_k  \beta_k \geqslant \frac{1}{2}+ \sigma }} 1 \ \  d\sigma.
				\end{equation*}
				
				Denote the inner sum by $N^*(\sigma; T)$. Since the quantities $\rho_k$ with $1 \leqslant k \leqslant 2m$ enter the last sum symmetrically, it suffices to estimate the contribution from the terms corresponding to $\beta_1 = \max_k \beta_k \geqslant \frac{1}{2} + \sigma$. We have
				\begin{equation*}
					N^*(\sigma; T)  \leqslant  \sum\limits_{\substack{0<\gamma_1 \leqslant T \\ \beta_1 \geqslant \frac{1}{2}+\sigma}} \sum\limits_{\substack{0<\gamma_k \leqslant T \\ 1 < k < 2m }} \sum\limits_{\substack{0 < \gamma_{2m} \leqslant T \\ |\Delta|\leqslant 1}} 1.
				\end{equation*}
				Let $\Delta' = \gamma_1 + \cdots + \gamma_m - \gamma_{m+1} - \cdots - \gamma_{2m-1}$. Since the inequality $|\Delta| \leqslant 1$ implies the inequality $|\gamma_{2m} - |\Delta'|| \leqslant 1$, we obtain
				\[
				\sum\limits_{\substack{0<\gamma\leqslant T \\ |\Delta|\leqslant 1}} 1 \leqslant \sum\limits_{\substack{0<\gamma\leqslant T \\ |\Delta'|-1 \leqslant \gamma \leqslant |\Delta'|+1}} 1 \ll_m \log T.
				\]
				Hence, we find that
				\begin{multline*}
					N^*(\sigma; T) \ll_m N(T)^{2m-2} N \left( \dfrac{1}{2}+\sigma; T\right) \log T \ll\\
					T^{2m-2} (\log T)^{2m-1} N \left( \dfrac{1}{2}+\sigma; T\right) \ll \\  T^{2m-2} (\log T)^{2m-1} T^{1-\frac{\sigma}{4}} \log T \ll  T^{2m-1} (\log T)^{2m} T^{-\frac{\sigma}{4}}.
				\end{multline*}
				So, $N^*(\sigma; T) \ll_m T^{2m-1} (\log T)^{2m} T^{-\frac{\sigma}{4}}$. Therefore,
				\begin{multline*}
					J_1 \ll_m U^{4m+1} \int \limits_0^{U^{-1}} \sigma^{4m-1} N^*(\sigma; T) d\sigma \ll U^{4m+1} T^{2m-1} (\log T)^{2m} \int \limits_0^{U^{-1}} \sigma^{4m-1} T^{-\frac{\sigma}{4}} d\sigma.
				\end{multline*}
				The integral does not exceed
				\[
				\int\limits_0^{+\infty} \sigma^{4m-1} e^{-\frac{\sigma}{4}\log T} d\sigma \ll_m \dfrac{1}{\left(\log T \right)^{4m} }.
				\]
				Thus, 
				\[
				J_1 \ll_m \dfrac{T^{2m-1}U^{4m+1}}{\left(\log T \right)^{2m}}.
				\]
				
				Passing to the estimate of $J_2$ and using the inequality $|2 - e^x - e^{-x}| \ll e^x$, we have
				\[
				|I(\xi)| = \prod\limits_{k=1}^{2m} \left|  2-e^{2\pi \xi \left( \beta_k - \frac{1}{2}\right) }-e^{-2\pi \xi \left( \beta_k - \frac{1}{2}\right)}  \right| \ll \left(e^{2\pi \xi \alpha} \right)^{2m} \leqslant e^{4\pi m U \alpha}.
				\]
				From here we obtain
				\[
				\int\limits_0^U I(\xi) d\xi \ll U e^{4\pi m U \alpha}
				\]
				and
				\[
				J_2 = \sum\limits_{\substack{0< \gamma_k \leqslant T \\ 1 \leqslant k \leqslant 2m \\ |\Delta| \leqslant 1,\,\alpha > U^{-1}}}\int\limits_0^U I(\xi) d\xi \ll U \sum\limits_{\substack{0< \gamma_k \leqslant T \\ 1 \leqslant k \leqslant 2m \\ |\Delta| \leqslant 1,\,\alpha > U^{-1}}} e^{4\pi m U \alpha}.
				\]
				Since
				\[
				e^{4\pi m U \alpha} = 4 \pi m U \int\limits_{U^{-1}}^\alpha e^{4\pi m U \sigma} d\sigma + e^{4\pi m}\ll_m U \int\limits_{U^{-1}}^\alpha e^{4\pi m U \sigma} d\sigma + 1,
				\]
				it follows that
				\[
				J_2 \ll_m U^2 \sum\limits_{\substack{0< \gamma_k \leqslant T \\ 1 \leqslant k \leqslant 2m \\ |\Delta| \leqslant 1,\, \alpha > U^{-1}}}  \int\limits_{{U}^{-1}}^\alpha e^{4\pi m U \sigma} d\sigma  + U \sum\limits_{\substack{0< \gamma_k \leqslant T \\ 1 \leqslant k \leqslant 2m \\ |\Delta| \leqslant 1,\, \alpha > U^{-1}}} 1 = J_2^{(1)} + J_2^{(2)},
				\]
				where the meaning of the notations $J_2^{(1)}$ and $J_2^{(2)}$ is obvious. Now we estimate $J_2^{(1)}$. First of all, we have
				\begin{multline*}
					J_2^{(1)} = U^2 \int\limits_{U^{-1}}^{{1}/{2}} e^{4\pi m U \sigma} \sum\limits_{\substack{0< \gamma_k \leqslant T \\ 1 \leqslant k \leqslant 2m \\ |\Delta| \leqslant 1,\,\alpha  \geqslant \sigma }} 1 \ \ d\sigma \leqslant U^2 \int\limits_{{U}^{-1}}^{{1}/{2}} e^{4\pi m U \sigma} N^*(\sigma; T) d\sigma \ll_m \\   U^2 \int\limits_{{U}^{-1}}^{+\infty} e^{4\pi m U \sigma}  T^{2m-1} (\log T)^{2m} T^{-\frac{\sigma}{4}} d\sigma \ll U^2 T^{2m-1}  (\log T)^{2m} \int\limits_{{U}^{-1}}^{+\infty} e^{4\pi m U \sigma - \frac{\sigma}{4}\log T} d\sigma.
				\end{multline*}
				Using the inequality $U \leqslant \frac{\log T}{32 \pi m}$, for the last integral we obtain
				\begin{equation*}
					\int\limits_{U^{-1}}^{+\infty} e^{\left(4\pi m U   - \frac{\log T}{4}\right)\sigma } d\sigma =  \int\limits_{\frac{\log T}{4U}-4\pi m}^{+\infty} e^{-w} \dfrac{dw}{\frac{\log T}{4}-4\pi m U} \ll_m  \dfrac{T^{-\frac{1}{4U}}}{\log T}.
				\end{equation*}
				Hence,
				\[
				J_2^{(1)} \ll_m \dfrac{U^2T^{2m-1}(\log T)^{2m}T^{-\frac{1}{4U}}}{\log T} \ll U^2 T^{2m-1} (\log T)^{2m-1} T^{-\frac{1}{4U}}.
				\]
				We now estimate $J_2^{(2)}$. We have
				\begin{multline*}
					$$
					J_2^{(2)} = U  \sum\limits_{\substack{0< \gamma_k \leqslant T \\ 1 \leqslant k \leqslant 2m \\ |\Delta| \leqslant 1,\,\alpha > U^{-1}}} 1 \leqslant U  \sum\limits_{\substack{0< \gamma_k \leqslant T \\ 1 \leqslant k \leqslant 2m \\ |\Delta| \leqslant 1 \\   \max\limits_{k} \beta_k  \geqslant \frac{1}{2} + \frac{1}{U} }} 1 = UN^*\left( \dfrac{1}{U}; T\right) \ll \\ UT^{2m-1} (\log T)^{2m} T^{-\frac{1}{4U}}.
					$$
				\end{multline*}
				Thus,
				\[
				J_2 \ll_m  UT^{2m-1} (\log T)^{2m} T^{-\frac{1}{4U}}.
				\]
				
				Now we estimate the quantity $J_3$. We have
				\begin{multline*}
					$$
					I(\xi) = e^{2\pi i \xi \Delta} \prod\limits_{k=1}^{2m} \left( 2 - 2\cosh \left( 2\pi \xi \left(\beta_k -\dfrac{1}{2} \right) \right)  \right)  = \\ e^{2\pi i \xi \Delta} \left( 2^{2m} + \sum\limits_{s=1}^{2m} 2^{2m-s} (-1)^s \sum\limits_{1 \leqslant k_1 < \cdots < k_s\leqslant 2m} \prod\limits_{\ell=1}^s 2\cosh\left(  2\pi \xi \left(\beta_{k_\ell} -\dfrac{1}{2} \right) \right) \right).
					$$
				\end{multline*}
				Using Lemma \ref{L3}, we obtain
				\begin{multline*}
					I(\xi) = e^{2\pi i \xi \Delta} \left(  2^{2m} + \sum\limits_{s=1}^{2m} 2^{2m-s} (-1)^s \sum\limits_{1 \leqslant k_1 < \cdots < k_s\leqslant 2m} \right. \\ \left. \sum\limits_{\substack{\varepsilon_{k_j} \in \{-1, +1\},\, 2 \leqslant j \leqslant s}} 2\cosh\left(  2\pi \xi \left(  \left(\beta_{k_1} -\dfrac{1}{2} \right) + \varepsilon_{k_2}  \left(\beta_{k_2} -\dfrac{1}{2} \right) + \cdots + \varepsilon_{k_s}  \left(\beta_{k_s} -\dfrac{1}{2} \right) \right)\right)  \right).
				\end{multline*}
				Set 
				\begin{equation}\label{g(z)}
					g(z) = 2 \int\limits_0^U e^{2\pi i \xi \Delta} \cosh(2 \pi \xi z) d\xi,
				\end{equation}
				then we will have
				\begin{multline}\label{intI}
					\int\limits_0^U I(\xi) d\xi = 2^{2m-1} g(0) + \sum\limits_{s=1}^{2m} 2^{2m-s}  (-1)^s \sum\limits_{1 \leqslant k_1 < \cdots < k_s\leqslant 2m}\\  \sum\limits_{\substack{\varepsilon_{k_j} \in \{-1, +1\} \\ 2 \leqslant j \leqslant s}} g \left( \beta_{k_1} - \dfrac{1}{2} +\varepsilon_{k_2} \left(  \beta_{k_2} - \dfrac{1}{2}\right) + \cdots +  \varepsilon_{k_s} \left(  \beta_{k_s} - \dfrac{1}{2}\right)    \right). 
				\end{multline}
				Let us expand $g(z)$ in a power series in $z$. We have
				\[
				g(z) = 2 \int\limits_0^U e^{2\pi i \xi \Delta} \left(\sum\limits_{0 \leqslant k \leqslant 2m-1} \dfrac{(2\pi \xi z)^{2k}}{(2k)!}  + \sum\limits_{k \geqslant 2m}  \dfrac{(2\pi \xi z)^{2k}}{(2k)!}      \right) d\xi = F_m(z) + R_m(z),
				\]
				where
				\[
				F_m(z) = \sum\limits_{0 \leqslant k \leqslant 2m-1} \varkappa_k z^{2k},
				\]
				\[
				R_m(z) = \sum\limits_{k \geqslant 2m} \varkappa_k z^{2k}
				\]
				and
				\[
				\varkappa_k = \varkappa_k (U; \Delta) = 2 \dfrac{(2\pi )^{2k}}{(2k)!} \int\limits_0^U e^{2\pi i \xi \Delta} \xi^{2k} d\xi.
				\]
				Let us determine the contribution of the quantity $F_m(z)$ to formula \eqref{intI}. The term $\varkappa_0$ gives a contribution equal to
				\begin{multline*}
					\varkappa_0 \left(2^{2m-1} + \sum\limits_{s=1}^{2m} 2^{2m-s}(-1)^s \binom{2m}{s} 2^{s-1} \right) = \\  \varkappa_0 2^{2m-1}\left(1 +  \sum\limits_{s=1}^{2m}(-1)^s \binom{2m}{s} \right) = \varkappa_0  2^{2m-1} (1-1)^{2m} = 0.
				\end{multline*}
				The contribution from the term $\varkappa_k z^{2k}$ with $1 \leqslant k < 2m$ is zero by Lemma 2. Thus, the contribution of $F_m(z)$ to \eqref{intI} is zero.
				
				Now we find the contribution of the quantity $R_m(z)$ to \eqref{intI}. We have
				\[
				\int\limits_0^U  e^{2\pi i \xi \Delta} \xi^{2k} d\xi = \int\limits_0^U \xi^{2k} \dfrac{d  e^{2\pi i \xi \Delta}}{2\pi i \Delta} = \dfrac{U^{2k} e^{2\pi i U \Delta}}{2\pi i \Delta} - \dfrac{2k}{2\pi i \Delta} \int\limits_0^U  e^{2\pi i \xi \Delta} \xi^{2k-1} d\xi \ll \dfrac{U^{2k}}{\Delta}.
				\]
				Setting
				\[
				z = \beta_{k_1} - \dfrac{1}{2} + \varepsilon_{k_2} \left( \beta_{k_2} -\dfrac{1}{2}\right) + \cdots +  \varepsilon_{k_s} \left( \beta_{k_s} -\dfrac{1}{2}\right),
				\]
				we obtain $|z| \leqslant 2m\alpha$. Hence, for such $z$ we have
				%\[
				%|z| \leqslant 2m \alpha \leqslant \dfrac{2m}{U} \leqslant \dfrac{2}{3};
				%\]
				%\[
				%U \geqslant 3m;
				%\]
				\begin{multline*}
					R_m(z) \ll \sum\limits_{k \geqslant 2m} \dfrac{(2\pi)^{2k}|z|^{2k}}{(2k)!} \dfrac{U^{2k}}{\Delta} \leqslant \dfrac{1}{\Delta}\sum\limits_{k \geqslant 2m} \dfrac{(8\pi m)^{2k}}{(2k)!}\left(\dfrac{\alpha U}{2}\right)^{2k}\leqslant\\
					\dfrac{1}{\Delta}\left(\dfrac{\alpha U}{2}\right)^{4m}\sum\limits_{k \geqslant 2m} \dfrac{(8\pi m)^{2k}}{(2k)!}\ll_m
					\dfrac{U^{4m}\alpha^{4m}}{\Delta}.
				\end{multline*}
				Thus, the contribution of $R_m(z)$ to formula \eqref{intI} does not exceed $$O_m\left(\frac{U^{4m}\alpha^{4m}}{\Delta} \right).$$ Therefore,
				\[
				\int\limits_0^U I(\xi) d\xi \ll \dfrac{\alpha^{4m}U^{4m}}{|\Delta|}
				\]
				and
				\begin{equation*}
					J_3 \ll \sum\limits_{\substack{0< \gamma_k \leqslant T \\ 1 \leqslant k \leqslant 2m \\ |\Delta| > 1,\, \alpha \leqslant U^{-1}}}  \dfrac{\alpha^{4m}U^{4m}}{\Delta} \ll U^{4m} \sum\limits_{\substack{j \geqslant 0 \\ 2^{j+1} \leqslant 2m T}} \dfrac{1}{2^j} \sum\limits_{\substack{0< \gamma_k \leqslant T \\ 1 \leqslant k \leqslant 2m \\ 2^{j} \leqslant |\Delta| < 2^{j+1} \\  \alpha \leqslant U^{-1}}} \alpha^{4m}.
				\end{equation*}
				Set $L = \log(2mT)/\log 2$. Then we have
				\begin{multline*}
					J_3\ll_m U^{4m} \sum\limits_{\substack{0\leqslant j \leqslant L}} \dfrac{1}{2^j} \sum\limits_{\substack{0< \gamma_k \leqslant T \\ 1 \leqslant k \leqslant 2m \\ 2^{j} < |\Delta| \leqslant  2^{j+1} \\  \alpha \leqslant U^{-1}}} \int\limits_0^{\alpha} \sigma^{4m-1} d\sigma \ll_m\\  U^{4m} \sum\limits_{\substack{ 0\leqslant j\leqslant L}} \dfrac{1}{2^j}  \int\limits_0^{U^{-1}} \sigma^{4m-1} \sum\limits_{\substack{0< \gamma_k \leqslant T \\ 1 \leqslant k \leqslant 2m \\ 2^{j} < |\Delta| \leqslant  2^{j+1},\,\alpha \geqslant \sigma}} 1 \ \ d\sigma \ll \\
					U^{4m} \sum\limits_{0 \leqslant j \leqslant L} \dfrac{1}{2^j}  \int\limits_0^{U^{-1}} \sigma^{4m-1} N_j(\sigma; T) d\sigma,
				\end{multline*}
				where \[
				N_j(\sigma; T) = \sum\limits_{\substack{0< \gamma_k \leqslant T \\ 1 \leqslant k \leqslant 2m \\ 2^{j} < |\Delta| \leqslant  2^{j+1} \\  \max\limits_{1 \leqslant k \leqslant 2m} \beta_k \geqslant \frac{1}{2}+ \sigma}} 1.
				\]
				Let us estimate the quantity $N_j(\sigma; T)$. We have
				\[
				N_j(\sigma; T) \ll_m \sum\limits_{\substack{0< \gamma_1 \leqslant T \\ \beta_1 \geqslant \frac{1}{2}+\sigma}}  \sum\limits_{\substack{0< \gamma_k \leqslant T \\ 1 < k < 2m}}  \sum\limits_{\substack{0< \gamma_{2m} \leqslant T \\ 2^{j} < |\Delta| \leqslant  2^{j+1}}}  1.
				\]
				For the inner sum we obtain
				\[
				\sum\limits_{\substack{0< \gamma_{2m} \leqslant T \\ 2^{j} < |\Delta| \leqslant  2^{j+1}}} 1 \ll \sum\limits_{\nu \leqslant 2^{j+1}} \sum\limits_{\substack{\gamma_{2m}: \\ \nu < |\Delta| < \nu+1}} 1 \ll 2^j \log T.
				\]
				Therefore,
				\[
				N_j(\sigma; T) \ll 2^j (\log T) N \left(\dfrac{1}{2}+\sigma; T\right) N(T)^{2m-2} \ll 2^j (\log T)^{2m} T^{2m-1} T^{-\frac{\sigma}{4}}
				\]
				and
				\begin{multline*}
					J_3 \ll U^{4m} \sum\limits_{0 \leqslant j \leqslant L} (\log T)^{2m} T^{2m-1} \int\limits_0^{U^{-1}} \sigma^{4m-1} T^{-\frac{\sigma}{4}} d\sigma \ll_m \\ U^{4m}  (\log T)^{2m+1}  T^{2m-1} \int\limits_0^{+\infty} w ^{4m-1} e^{-w} \dfrac{d w}{\left( \log T\right) ^{4m}} \ll_m \dfrac{T^{2m-1}U^{4m}}{(\log T)^{2m-1}}.
				\end{multline*}
				Thus,
				$$J_3\ll_m \dfrac{T^{2m-1}U^{4m}}{(\log T)^{2m-1}}.$$
				
				Finally, we estimate the quantity $J_4$. From \eqref{g(z)} we obtain
				%\[
				%J_4 =   \sum\limits_{\substack{0< \gamma_k \leqslant T \\ 1 \leqslant k \leqslant 2m \\ |\Delta| > 1 \\  \alpha > U^{-1}}} \int\limits_0^U I(\xi) d\xi;
				%\]
				\[
				g(z) = \dfrac{e^{2\pi Ui \Delta}}{2\pi} \left(\dfrac{e^{2\pi U z}}{i\Delta + z} + \dfrac{e^{-2\pi U z}}{i\Delta - z} \right)  - \dfrac{1}{2\pi} \left( \dfrac{1}{i\Delta + z}+ \dfrac{1}{i\Delta - z}\right).
				\]
				Hence, for real $z$ we obtain
				\[
				g(z) \ll \dfrac{e^{2\pi U |z|}+1}{\sqrt{\Delta^2 + z^2}} \ll \dfrac{e^{2\pi U |z|}}{|\Delta|}.
				\]
				For
				\[
				z= \beta_{k_1} - \dfrac{1}{2} + \varepsilon_{k_2}\left(\beta_{k_2} - \dfrac{1}{2} \right) + \cdots + \varepsilon_{k_s}\left(\beta_{k_s} - \dfrac{1}{2} \right) 
				\]
				\[
				(1 \leqslant k_1 < \cdots < k_s \leqslant 2m, \ \ 1 \leqslant s \leqslant 2m)
				\]
				we have $|z| \leqslant s \alpha \leqslant 2 m \alpha$. Therefore,
				\[
				g(z) \ll \dfrac{e^{4\pi m U \alpha}}{|\Delta|}.
				\]
				From \eqref{intI}, we have
				\[
				\int\limits_0^U I(\xi) d \xi \ll_m \dfrac{e^{4\pi m U \alpha}}{|\Delta|}.
				\]
				Hence,
				\begin{multline*}
					J_4 \ll_m  \sum\limits_{\substack{0< \gamma_k \leqslant T \\ 1 \leqslant k \leqslant 2m \\ |\Delta| > 1,\, \alpha > U^{-1}}} \dfrac{e^{4\pi m U \alpha}}{|\Delta|} \ll \sum\limits_{0 \leqslant j \leqslant L} \dfrac{1}{2^j}  \sum\limits_{\substack{0< \gamma_k \leqslant T \\ 1 \leqslant k \leqslant 2m \\ 2^j < |\Delta| \leqslant 2^{j+1} \\  \alpha > U^{-1}}} e^{4\pi m U \alpha} = \\ \sum\limits_{0 \leqslant j \leqslant L} \dfrac{1}{2^j}  \sum\limits_{\substack{0< \gamma_k \leqslant T \\ 1 \leqslant k \leqslant 2m \\ 2^j < |\Delta| \leqslant 2^{j+1} \\  \alpha > U^{-1}}} \left(4\pi m U \int\limits_{U^{-1}}^\alpha e^{4\pi U \sigma m} d\sigma + e^{4\pi m} \right) = J_4^{(1)} + J_4^{(2)}.
				\end{multline*}
				First, we estimate $J_4^{(1)}$. We have
				\begin{multline*}
					J_4^{(1)} \ll_m U \sum\limits_{0 \leqslant j \leqslant L} \dfrac{1}{2^j} \sum\limits_{\substack{0< \gamma_k \leqslant T \\ 1 \leqslant k \leqslant 2m \\ 2^j < |\Delta| \leqslant 2^{j+1} \\  \alpha > U^{-1}}}   \int\limits_{U^{-1}}^\alpha e^{4\pi U \sigma m} d\sigma =\\
					U \sum\limits_{0 \leqslant j \leqslant L} \dfrac{1}{2^j} \int\limits_{U^{-1}}^{{1}/{2}}  e^{4\pi U \sigma m}  \sum\limits_{\substack{0< \gamma_k \leqslant T \\ 1 \leqslant k \leqslant 2m \\ 2^j < |\Delta| \leqslant 2^{j+1} \\  \alpha > \sigma}} 1 \ \ d\sigma \ll \\
					U \sum\limits_{0 \leqslant j \leqslant L} \dfrac{1}{2^j} \int\limits_{U^{-1}}^{{1}/{2}} e^{4\pi U \sigma m} N_j(\sigma; T) d\sigma \ll_m \\ U(\log T)^{2m+1} T^{2m-1} \int\limits_{U^{-1}}^{+\infty} e^{4\pi U \sigma m-\frac{\sigma}{4}\log T}  d\sigma \ll_m U(\log T)^{2m} T^{2m-1} T^{-\frac{1}{4U}}.
				\end{multline*}
				Now we estimate $J_4^{(2)}$. We have
				\begin{multline*}
					J_4^{(2)} \ll_m \sum\limits_{0 \leqslant j \leqslant L} \dfrac{1}{2^j} N_j\left(\dfrac{1}{U}; T \right) \ll  \sum\limits_{0 \leqslant j \leqslant L}  (\log T)^{2m} T^{2m-1} T^{-\frac{1}{4U}} \ll_m \\ T^{2m-1} (\log T)^{2m+1}T^{-\frac{1}{4U}}.
				\end{multline*}
				Hence, we get
				$$J_4 \ll_m T^{2m-1} (\log T)^{2m+1}T^{-\frac{1}{4U}}.$$
				Thus,
				$$
				J \leqslant J_1+J_2+J_3+J_4\ll T^{2m-1} \left(\dfrac{  U^{4m}}{(\log T)^{2m-1}} + (\log T)^{2m+1} T^{-\frac{1}{4U}} \right).
				$$
				This concludes the proof.
			\end{proof}
			
			\section*{Proof of Theorem \ref{Th1}}
			
			By applying the inverse Fourier transform, we have
			\begin{equation*}
				H =  \sum\limits_{\substack{0 < \gamma_k \leqslant T \\ 1 \leqslant k \leqslant m}} \int\limits_{-\infty}^{+\infty} \hat{h}(\xi) e^{2\pi i \xi \left(a_1\gamma_1 + a_2\gamma_2 + \cdots + a_m\gamma_m \right) } d \xi.
			\end{equation*}
			Since the integral converges uniformly, it follows that
			\[
			H = \int\limits_{-\infty}^{+\infty} \hat{h}(\xi) \sum\limits_{0 < \gamma_1 \leqslant T } e^{2 \pi i \xi a_1\gamma_1} \cdots  \sum\limits_{0 < \gamma_m \leqslant T } e^{2 \pi i \xi a_m\gamma_m}  d\xi.
			\]
			Set $Q(\xi) =  \sum\limits_{0 < \gamma  \leqslant T }e^{2 \pi i \xi  \gamma }  $ and \[
			P_m(\xi) = P_m(\xi; a_1, \ldots, a_m) = \prod\limits_{k=1}^m Q(a_k\xi).
			\]
			Then using the equality $\hat{h}(-\xi) = \overline{\hat{h}(\xi) }$, we get
			\[
			H = \int\limits_{-\infty}^{+\infty} \hat{h}(\xi) P_m(\xi) d\xi = 2 \Re \int\limits_{0}^{+\infty} \hat{h}(\xi)P_m(\xi) d\xi.
			\]
			Put
			$$U= \dfrac{\log T}{100m \log \log T}\ \ \ (T \geqslant T_0).$$ Then we have
			\[
			H = 2 \Re \int\limits_{0}^{U} \hat{h}(\xi) P_m(\xi)d\xi + 2 \Re \int\limits_{U}^{+\infty} \hat{h}(\xi) P_m(\xi)d\xi =  H_1 + H_2.
			\]
			Let us estimate the integral $H_2$. We have
			\begin{multline*}
				|H_2| \leqslant 2 \left| \int\limits_{U}^{+\infty} \hat{h}(\xi) P_m(\xi)d\xi \right| 	= 2 \left| \int\limits_{U}^{+\infty} \hat{h}(\xi) d\left( \int\limits_{0}^{\xi} P_m(t)dt  \right)  \right| =  \\ 2 \left|\lim\limits_{\xi \to +\infty}\left(   \hat{h}(\xi)  \int\limits_{0}^{\xi}P_m(t)dt\right)  - \hat{h}(U)\int\limits_{0}^{U}P_m(t)dt - \int\limits_{U}^{+\infty} \left(\int\limits_{0}^{\xi}P_m(t)dt \right)  \hat{h'}(\xi)d\xi \right|.
			\end{multline*}
			For $K \geqslant 1$, by Lemma \ref{L5}, we have
			\begin{multline*}
				\int\limits_{0}^{K} P_m(\xi)d\xi = \int\limits_{0}^{K} \prod\limits_{k=1}^m Q(a_k\xi)d\xi = \int\limits_{0}^{K} \sum\limits_{\substack{0 < \gamma_k \leqslant T \\ 1 \leqslant k \leqslant m}} e^{2\pi i \xi \left(a_1\gamma_1 +   \cdots + a_m\gamma_m \right) }d\xi = \\
				\sum\limits_{\substack{0 < \gamma_k \leqslant T \\ 1 \leqslant k \leqslant m}} \int\limits_{0}^{K} e^{2\pi i \xi \Delta} d\xi \ll \sum\limits_{\substack{0 < \gamma_k \leqslant T \\ 1 \leqslant k \leqslant m}} \min \left( K, \dfrac{1}{|\Delta|} \right) \ll T^{m-1} (\log T)^m (K+\log T),
			\end{multline*}
			where $\Delta =a_1\gamma_1 + \cdots + a_m\gamma_m $. 
			
			Hence, since $\hat{h}(\xi) \ll \dfrac{1}{\xi^{m+2}}$ (see \cite[Lemma 1]{Iud2023}), it follows that
			\[
			\hat{h}(\xi) \int\limits_{0}^{\xi} P_m(t) dt \ll \dfrac{T^{m-1}(\log T)^m(\xi + \log T)}{\xi^{m+2}} \to 0
			\]
			as $\xi \to +\infty$. 
			
			Similarly, we obtain
			\begin{multline*}
				\hat{h}(U)  \int\limits_{0}^{U} P_m(t)dt \ll \dfrac{T^{m-1}(\log T)^m(U + \log T)}{U^{m+2}} \ll_m \\ \dfrac{T^{m-1}(\log T)^{m+1}}{(\log T)^{m+2}} (\log \log T)^{m+2} =  	\dfrac{T^{m-1}(\log \log T)^{m+2}}{\log T}.
			\end{multline*}	
			Finally, using the estimate $\hat{h'}(\xi) \ll \dfrac{1}{\xi^{m+3}}$, we get
			\begin{multline*}
				\int\limits_{U}^{+\infty} \left( \int\limits_{0}^{\xi}P_m(t)dt  \right) \hat{h'}(\xi) d\xi \ll  \int\limits_{U}^{+\infty} \dfrac{T^{m-1}(\log T)^m(\xi + \log T)}{\xi^{m+3}} d\xi \ll_m \\
				\dfrac{T^{m-1}(\log T)^m}{U^{m+1}}  + \dfrac{T^{m-1}(\log T)^{m+1}}{U^{m+2}} \ll_m \dfrac{T^{m-1}(\log \log  T)^{m+2}}{\log T}.
			\end{multline*}		
			Thus, we obtain
			\[
			H = H_1 + O_m \left( T^{m-1}\dfrac{(\log \log  T)^{m+2}}{\log T}\right).
			\]	
			Let us transform now $H_1$. We have
			\[
			Q(\xi) = \sum\limits_{0 < \gamma  \leqslant T } e^{2\pi i \gamma \xi} = \sum\limits_{0 < \gamma  \leqslant T } \left( e^{2\pi \xi \left(\rho-\frac{1}{2} \right) }  +e^{2\pi i \xi \gamma } - e^{2\pi \xi \left(\rho-\frac{1}{2} \right) }  \right).
			\]	
			Setting $x=e^{2\pi \xi}$, we will have \[
			Q(\xi) = \dfrac{1}{\sqrt{x}}  \sum\limits_{0 < \gamma  \leqslant T }x^{\rho} +  \sum\limits_{0 < \gamma  \leqslant T } e^{2\pi i \xi \gamma } \left(1 - e^{2\pi \xi \left(\beta-\frac{1}{2} \right) } \right) = \dfrac{1}{\sqrt{x}}   \sum\limits_{0 < \gamma  \leqslant T } x^{\rho}  + D(\xi).
			\]
			As is known (see the remark to Lemma 2.1 in the paper \cite{Ford-Zahar2015}),
			\[
			\sum\limits_{0 < \gamma  \leqslant T } x^{\rho} = - \dfrac{\Lambda(n(x))}{2\pi} \dfrac{e^{iT\log \frac{x}{n(x)}}-1}{i \log \frac{x}{n(x)}} + O \left(x^{\frac{1}{2}+\varepsilon} + \min \left(\dfrac{\log T}{\log x}, T\log T\right) \right).
			\]
			Hence, we find
			\begin{multline}\label{Q}
				Q(\xi) = - \dfrac{\Lambda(n(x))}{2\pi\sqrt{n(x)}} \dfrac{e^{iT\log \frac{x}{n(x)}}-1}{i \log \frac{x}{n(x)}} +   \dfrac{\Lambda(n(x))}{2\pi} \dfrac{e^{iT\log \frac{x}{n(x)}}-1}{i \log \frac{x}{n(x)}} \left( \dfrac{1}{\sqrt{n(x)}}-\dfrac{1}{\sqrt{x}} \right) + \\ + O 	 \left(x^{\frac{1}{2}+\varepsilon}\log T + \min \left(\dfrac{\log T}{\log x}, T\log T\right) \right) + D(\xi). 
			\end{multline}
			Let us denote the first term in \eqref{Q} by $M(\xi)$, and the sum of the second and third by $E(\xi)$. Then
			\[
			Q(\xi) = M(\xi) + E(\xi) + D(\xi)
			\]
			and 
			\[
			H_1 = 2 \Re \int\limits_0^U \hat{h}(\xi) \prod\limits_{\substack{k=1 \\ a_k>0}}^m M(\xi a_k) \prod\limits_{\substack{k=1 \\ a_k<0}}^m \overline{M}(\xi |a_k|) d\xi + R_1,
			\]
			where $$R_1 = 2 \Re 	\int\limits_0^U  \hat{h}(\xi) \left(P_m(\xi) -  \prod\limits_{\substack{k=1 \\ a_k>0}}^m M(\xi a_k) \prod\limits_{\substack{k=1 \\ a_k<0}}^m \overline{M}(\xi |a_k|) \right)  d\xi.$$
			
			Now we are going to estimate the value $R_1$. Let $b_1, b_2, \ldots, b_q$ be the positive elements of the tuple $\left(a_1, \ldots, a_m\right)$, and let $-c_1, -c_2, \ldots, -c_\ell$ be the negative elements of this tuple, respectively. Then  
			$$
			b_1 + b_2 + \cdots + b_q = c_1 + c_2 + \cdots + c_\ell
			$$
			and $\ell + q = m$. From this we obtain
			\begin{multline*}
				\left|P_m(\xi) -  \prod\limits_{\substack{k=1 \\ a_k>0}}^m M(\xi a_k) \prod\limits_{\substack{k=1 \\ a_k<0}}^m \overline{M}(\xi |a_k|) \right|  =\\
				\left| \prod\limits_{t=1}^q Q(\xi b_t)  \prod\limits_{s=1}^\ell \overline{Q}(\xi c_s) - \prod\limits_{\substack{k=1 \\ a_k>0}}^m M(\xi a_k) \prod\limits_{\substack{k=1 \\ a_k<0}}^m \overline{M}(\xi |a_k|)  \right|  = \\
				\left| \prod\limits_{t=1}^q \left(M(\xi b_t) + E(\xi b_t) + D(\xi b_t) \right)     \prod\limits_{s=1}^\ell \left(\overline{M}(\xi c_s) + \overline{E}(\xi c_s) + \overline{D}(\xi c_s) \right) -\right.\\
				\left. \prod\limits_{t=1}^q M(\xi b_t) \prod\limits_{s=1}^\ell \overline{M}(\xi c_s) \right|.
			\end{multline*}
			Expanding the brackets and using the triangle inequality, we obtain
			\begin{multline*}
				\left|P_m(\xi) - \prod\limits_{\substack{k=1 \\ a_k>0}}^m M(\xi a_k) \prod\limits_{\substack{k=1 \\ a_k<0}}^m \overline{M}(\xi |a_k|) \right| \leqslant\\
				\mathop{{\sum}'}\limits_{\substack{\varepsilon_{\ell k} \in \{0, 1\} \\ 1 \leqslant \ell \leqslant 3 \\ 1 \leqslant k \leqslant m}}  \prod\limits_{k=1}^m \left|M^{\varepsilon_{1k}} \left(\xi |a_k| \right) E^{\varepsilon_{2k}} \left(\xi |a_k| \right) D^{\varepsilon_{3k}} \left(\xi |a_k| \right) \right|,
			\end{multline*}
			where the prime indicates that the summation runs over $\varepsilon_{\ell k}$ satisfying the conditions $$\varepsilon_{1k} + \varepsilon_{2k} + \varepsilon_{3k}=1 \ \   (1 \leqslant k \leqslant m)$$  and $$\sum\limits_{k=1}^m  \left( \varepsilon_{2k} +  \varepsilon_{3k}\right)  > 0.$$
			So,
			\begin{equation}\label{R1}
				|R_1| \leqslant 2 \mathop{{\sum}'}\limits_{\substack{\varepsilon_{\ell k} \in \{0, 1\} \\ 1 \leqslant \ell \leqslant 3 \\ 1 \leqslant k \leqslant m}} \int\limits_0^U |\hat{h}(\xi)| \prod\limits_{k=1}^m \left|M^{\varepsilon_{1k}} \left(\xi |a_k| \right) E^{\varepsilon_{2k}} \left(\xi |a_k| \right) D^{\varepsilon_{3k}} \left(\xi |a_k| \right) \right| d\xi.
			\end{equation}
			For each fixed set $\{\varepsilon_{\ell k}\}_{\ell,k}$, the integral in \eqref{R1} has the form
			\[
			I = \int\limits_0^U |\hat{h}(\xi)|\cdot  \left|\prod\limits_{k=1}^m F_k(\xi|a_k|) \right| d\xi,
			\]
			where $F_k$ is one of the functions $M$, $E$, or $D$. Let us estimate the integral $I$. By applying the generalized H\"older inequality, we obtain
			\[
			I \leqslant \left(\int\limits_0^U |\hat{h}(\xi)| \cdot \left|F_1(\xi|a_1|) \right|^m d\xi  \right)^{\frac{1}{m}} \cdots  \left(\int\limits_0^U |\hat{h}(\xi)| \cdot \left|F_m(\xi|a_m|) \right|^m d\xi  \right)^{\frac{1}{m}}.
			\]
			We now estimate each of the resulting integrals.
			
			1) The case $F = E$. We need to estimate the integral
			\[
			I_E= \int\limits_0^U |\hat{h}(\xi)| \cdot |E(\xi a)|^m d\xi \  \ (a \geqslant 1).
			\]
			We will use the estimates obtained in the course of the proof of Lemma~1 and Theorem~1 of \cite{Iud2023}:
			%Нам понадобятся следующие оценки (см. лемму 1 и док-во теоремы 1 в \cite{Iud2023}):
			\[
			|\hat{h}(\xi)| \ll |\xi|,
			\]
			\[
			E(\xi) \ll e^{2\pi \xi} \log T + \min \left(\dfrac{\log T}{\xi}, T\log T \right).
			\]
			Using the inequality $(a+b)^m \ll_m a^m + b^m$, we get
			\begin{multline*}
				I_E \ll_m \int\limits_0^U |\hat{h}(\xi)| \left(e^{2\pi \xi a m}(\log T)^m + {\min}^m \left(\dfrac{\log T}{\xi}, T\log T \right) \right) d\xi \ll \\
				(\log T)^m 	 \int\limits_0^U  \xi e^{2\pi \xi am} d\xi +  \int\limits_0^U  \xi \min \left(\dfrac{(\log T)^m}{\xi^m}, T^m (\log T)^m \right) d\xi \ll \\
				U e^{2\pi U am} (\log T)^m + \left( \int\limits_0^{\frac{1}{T}} \xi d\xi\right) T^m (\log T)^m + \left( \int\limits_{\frac{1}{T}}^U \dfrac{ d\xi}{\xi^{m-1}}\right) (\log T)^m \ll_m\\
				T^{m-2} (\log T)^m.
			\end{multline*}
			Thus, $I_E \ll_m T^{m-2} (\log T)^m$.
			
			2) The case $F = D$. We need to estimate the integral
			\[
			I_D = \int\limits_0^U |\hat{h}(\xi)| \cdot |D(a\xi  )|^m d\xi.
			\]
			First, note that for any $2 \leqslant u \leqslant U$ and $k \geqslant 1$ we have
			\begin{equation}\label{ID}
				\int\limits_0^u |D(a\xi  )|^k d \xi \ll_k \dfrac{T^{k-1}u^{2k}}{(\log T)^{k-1}}. 
			\end{equation}
			Indeed, for even $k$ this follows from Lemma \ref{L6}. If $k = 2r + 1$ is odd, then the required estimate follows from the Cauchy~--~Schwarz inequality. Indeed,
			\begin{multline*}
				\int\limits_0^u  |D(a\xi  )|^k d\xi \leqslant \int\limits_0^{au} |D(\xi)|^k d\xi = \int\limits_0^{au} |D(\xi)|^{2r} \cdot |D(\xi)|  d\xi \leqslant\\ \sqrt{\int\limits_0^{au} |D(\xi)|^{4r} d\xi } \cdot  \sqrt{\int\limits_0^{au} |D(\xi)|^{2} d\xi } \ll_k 
				\sqrt{\dfrac{T^{4r-1}u^{8r}}{(\log T)^{4r-1}}} \cdot \sqrt{\dfrac{Tu^4}{\log T}} = \dfrac{T^{k-1}\cdot u^{2k}}{(\log T)^{k-1}}.
			\end{multline*}
			From this it follows that for $I_D$ we have
			\begin{multline*}
				I_D = \int\limits_1^U |\hat{h(\xi)}| \cdot | D(a\xi)|^m d\xi + O \left(\dfrac{T^{m-1}}{(\log T)^{m-1}} \right) \ll\\
				\int\limits_1^U \dfrac{1}{\xi^{m+2}}  | D(a\xi)|^m d\xi + O \left(\dfrac{T^{m-1}}{(\log T)^{m-1}} \right) \ll_m \\
				\int\limits_1^{aU}  \dfrac{1}{\xi^{m+2}} | D(\xi)|^m d\xi + \dfrac{T^{m-1}}{(\log T)^{m-1}} = I_D' + \dfrac{T^{m-1}}{(\log T)^{m-1}}.
			\end{multline*}
			Integration by parts gives
			\begin{multline*}
				I_D' =  \int\limits_1^{aU}  \dfrac{1}{\xi^{m+2}}  d \left( \int\limits_1^{\xi} |D(t)|^m dt \right)  =\\
				\dfrac{1}{(aU)^{m+2}} \int\limits_1^{aU}  |D(t)|^m dt  + (m+2)  \int\limits_1^{aU} \left( \int\limits_1^\xi |D(t)|^m dt \right) \dfrac{d\xi}{\xi^{m+3}}.
			\end{multline*}
			From the inequality $m \geqslant 3$ and the estimate \eqref{ID}, we obtain
			%	Отсюда ввиду оценки \eqref{ID} и используя то, что $m \geqslant 3$, находим:
			\[
			I_D \ll_m \dfrac{1}{U^{m+2}} \dfrac{T^{m-1}U^{2m}}{(\log T)^{m-1}} + \int\limits_1^{aU} \dfrac{T^{m-1}\xi^{2m}}{(\log T)^{m-1}} \dfrac{d\xi}{\xi^{m+3}} \ll \dfrac{T^{m-1}}{(\log T)^{m-1}} U^{m-2} \ll \dfrac{T^{m-1}}{\log T}.
			\]
			Thus, we get
			\[
			I_D \ll_m \dfrac{T^{m-1}}{\log T}.
			\]
			
			3) The case $F = M$. We need to estimate the integral
			\[
			I_M = \int\limits_0^U |\hat{h}(\xi)| \cdot |M(a\xi)|^m d\xi \ll \int\limits_0^{aU} |M(\xi)|^m d\xi.
			\]
			We have
			\[
			I_M  \ll \dfrac{1}{(2\pi)^m}\int \limits_0^{aU} \dfrac{\Lambda(n(x))^m}{n(x)^{\frac{m}{2}}} \left| \dfrac{e^{iT \log \frac{x}{n(x)}}-1}{i \log \frac{x}{n(x)}}\right|^m d\xi.
			\]
			We split the integration interval into subintervals of the form
			\[
			\Delta_n = \left[\dfrac{1}{2\pi} \log \left(n-\dfrac{1}{2} \right), \dfrac{1}{2\pi} \log \left(n+\dfrac{1}{2} \right)  \right)  = [\alpha_n, \beta_n).
			\]
			Then, since $\xi \in \Delta_n\ (1 \leqslant n \leqslant e^{2\pi U a} + 1)$ holds if and only if $n(x) = n$, we obtain
			\begin{multline*}
				I_M \ll \dfrac{1}{(2\pi)^m} \sum\limits_{2 \leqslant n \leqslant e^{2\pi Ua} + 1}	\dfrac{\Lambda^m(n)}{n^{\frac{m}{2}}} \int\limits_{\alpha_n}^{\beta_n} \left|\dfrac{e^{iT\log \frac{x}{n}}-1}{i\log \frac{x}{n}} \right|^m d\xi\ll \\
				\dfrac{1}{(2\pi)^m} \sum\limits_{2 \leqslant n \leqslant e^{2\pi Ua} + 1}	\dfrac{\Lambda^m(n)}{n^{\frac{m}{2}}} \int\limits_{\log \left( 1-\frac{1}{2n}\right) }^{\log \left( 1+\frac{1}{2n}\right)} \left| \dfrac{e^{iTu}-1}{iu} \right|^m du.
			\end{multline*}
			%$u=\log \dfrac{x}{n_x}$, тогда $u = \log x - \log n$, $\log x = u + \log n$, $2 \pi \xi = u + \log n$, $\xi = \dfrac{u+\log n}{2\pi}$.
			Since 
			\[
			\left| \dfrac{e^{iTu}-1}{iu} \right| \leqslant \dfrac{2}{|u|}
			\]
			and
			\[
			\left| \dfrac{e^{iTu}-1}{iu} \right| = \left|\int\limits_0^T e^{iu \xi} d\xi \right| \leqslant T,
			\]
			it follows that
			\[
			I_M \ll \sum\limits_{2 \leqslant n \leqslant e^{2\pi Ua} + 1} \dfrac{\Lambda^m(n)}{n^{\frac{m}{2}}} \int\limits_{\log \left( 1-\frac{1}{2n}\right) }^{\log \left( 1+\frac{1}{2n}\right)} \min \left(T, \dfrac{1}{|u|} \right) ^m du.
			\]
			Using the estimate
			\begin{multline*}
				\int\limits_{\log \left( 1-\frac{1}{2n}\right) }^{\log \left( 1+\frac{1}{2n}\right)} {\min}^m \left(T, \dfrac{1}{|u|} \right) du =\\
				\left(\int\limits_{-\frac{1}{T}}^{\frac{1}{T}} + \int\limits_{|u|> \frac{1}{T}}\right)  {\min}^m \left(T, \dfrac{1}{|u|} \right) du \ll T^{m-1} + \int\limits_{\frac{1}{T}}^{+\infty} \dfrac{du}{u^m} \ll T^{m-1},
			\end{multline*}
			we get
			\[
			I_M \ll T^{m-1} \sum\limits_{n \geqslant 2} \dfrac{(\log n)^m}{n^{\frac{m}{2}}} \ll_m T^{m-1}.
			\]
			Since at least one of the functions $F_k$, $1 \leqslant k \leqslant m$, is different from $M$, we have
			\[
			I \ll_m T^{\frac{(m-1)^2}{m}} \left(\dfrac{T^{m-1}}{\log T} \right)^{\frac{1}{m}} = \dfrac{T^{m-1}}{(\log T)^{\frac{1}{m}}}.
			\]
			Hence,
			\[
			R_1 \ll_m  \dfrac{T^{m-1}}{(\log T)^{\frac{1}{m}}}
			\]
			and
			\begin{equation}\label{H1}
				H_1 = 2 \Re \int\limits_0^U \hat{h}(\xi) \prod\limits_{t=1}^q M(b_t \xi)  \prod\limits_{s=1}^\ell \overline{M}(c_s \xi) d \xi + O_m \left( \dfrac{T^{m-1}}{(\log T)^{\frac{1}{m}}} \right).
			\end{equation}
			
			%Пусть $\overrightarrow{a}:= (a_1, a_2, \ldots, a_m) \in (\mathbb{Z}^x)^m$. 
			%	Для некоторых целочисленных векторов $\textbf{c}$ и $\textbf{d}$ будем писать $\textbf{c} \sim \textbf{d}$, если их координаты совпадают с точности до перестановки. Тогда
			Let $d_1, d_2, \ldots, d_s$ be all distinct numbers from the tuple $(b_1, b_2, \ldots, b_q)$ with multiplicities $r_1, r_2, \ldots, r_s$, respectively. Similarly, let $-e_1, -e_2, \ldots, -e_v$ be the distinct numbers from the tuple $(-c_1, -c_2, \ldots, -c_\ell)$ with corresponding multiplicities $t_1, t_2, \ldots, t_v$. Then
			%\begin{multline}
			%\textbf{a} = (a_1, a_2, \ldots, a_m) \sim \left(b_1, b_2, \ldots, b_q, -c_1, -c_2, \ldots, -c_l \right) \sim  \\
			%	\left(\underbrace{d_1, \ldots, d_1}_{r_1}, \ldots, \underbrace{d_s, \ldots, d_s}_{r_s}, \underbrace{-e_1, \ldots, -e_1}_{t_1}, \ldots, \underbrace{-e_v, \ldots, -e_v}_{t_v} \right),
			%\end{multline}
			%$$b_1 + \cdots + b_q = c_1 + c_2 + \cdots + c_l,\ \ q+l = m,$$ 
			$$r_1 + \cdots + r_s = q,\ t_1 + \cdots + t_v = \ell,$$
			$$ \ r_\nu, t_\lambda \geqslant 1\ ( 1\leqslant \nu\leqslant s, 1\leqslant \lambda\leqslant v).$$ Since the numbers $d_1, \ldots, d_s$ and $-e_1, \ldots, -e_v$ are pairwise distinct, we obtain that, by assumption, they are pairwise coprime.  
			
			Let us denote the integral in \eqref{H1} by $J$. Then, setting $n = n(x)$, we have
			\begin{multline*}
				J = \sum\limits_{n \leqslant e^{2\pi U} + \frac{1}{2}} \int \limits_{\alpha_n^{''}}^{\beta_n^{''}} \hat{h} (\xi) \prod\limits_{\nu=1}^s (M(d_\nu \xi))^{r_\nu} \prod\limits_{\lambda=1}^v (\overline{M}(e_\lambda\xi))^{t_\lambda} d\xi = \\ \sum\limits_{n \leqslant e^{2\pi U} + \frac{1}{2}} \int \limits_{\alpha_n^{''}}^{\beta_n^{''}} \hat{h} (\xi) \prod\limits_{\nu=1}^s  \left( - \dfrac{\Lambda (n(x^{d_\nu}))}{2\pi \sqrt{n(x^{d_\nu})}} \cdot \dfrac{e^{i T \log \frac{x^{d_\nu}}{n(x^{d_\nu})}}-1}{i \log \frac{x^{d_\nu}}{n(x^{d_\nu})}}  \right)^{r_\nu} \times\\
				\prod\limits_{\lambda=1}^v \left(- \dfrac{\Lambda (n(x^{e_\lambda}))}{2\pi\sqrt{n(x^{e_\lambda})}} \cdot \dfrac{1-e^{-iT \log \frac{x^{e_\lambda}}{n(x^{e_\lambda})}}}{i \log \frac{x^{e_\lambda}}{n(x^{e_\lambda})}} \right)^{t_\lambda}d\xi,
			\end{multline*}
			where 
			\[
			\alpha_n^{''} = \max \left(0, \dfrac{1}{2\pi} \log \left(n-\dfrac{1}{2} \right)  \right),
			\]
			\[
			\beta_n^{''} = \min \left(U, \dfrac{1}{2\pi} \log \left(n+\dfrac{1}{2} \right)  \right).
			\] 
			Let $N_\nu = n(x^{d_\nu})$ for $1 \leqslant \nu \leqslant s$ and $M_\lambda = n(x^{e_\lambda})$ for $1 \leqslant \lambda \leqslant v$. Then from the inequalities
			\[
			n- \dfrac{1}{2} \leqslant x < n + \dfrac{1}{2},
			\]
			\[
			N_\nu  - \dfrac{1}{2} \leqslant x^{d_\nu} < N_\nu + \dfrac{1}{2},\ \ M_\lambda  - \dfrac{1}{2} \leqslant x^{e_\lambda} < M_\lambda + \dfrac{1}{2}
			\]
			we obtain
			\begin{equation}\label{N_nu}
				\left( n- \dfrac{1}{2}\right)^{d_\nu} -\dfrac{1}{2}<  N_\nu < \left( n+\dfrac{1}{2}\right)^{d_\nu} +\dfrac{1}{2},
			\end{equation}
			\begin{equation}\label{M_lambda}
				\left( n- \dfrac{1}{2}\right)^{e_\lambda}-\dfrac{1}{2} <  M_\lambda < \left( n + \dfrac{1}{2}\right)^{e_\lambda}+\dfrac{1}{2}. 
			\end{equation}
			Therefore,
			\[
			J = \sum\limits_{n \leqslant e^{2\pi U}+ \frac{1}{2}} \mathop{{\sum}'}\limits_{\substack{N_1, \ldots, N_s \\ M_1, \ldots, M_v}} \prod\limits_{\nu=1}^s \left( - \dfrac{\Lambda (N_\nu)}{2\pi \sqrt{N_\nu}}\right)^{r_\nu} \prod\limits_{\lambda=1}^v \left(- \dfrac{\Lambda(M_\lambda)}{2\pi \sqrt{M_\lambda}} \right)^{t_\lambda}  j(n),
			\]
			where the prime indicates summation under the conditions \eqref{N_nu} and \eqref{M_lambda}, and
			\[
			j(n) = j(n, N_1, \ldots, M_v) =  \int \limits_{\alpha_n^{'}}^{\beta_n^{'}} \hat{h}(\xi) \prod\limits_{\nu=1}^s \left(\dfrac{e^{iT \log \frac{x^{d_\nu}}{N_\nu}}-1}{i \log \frac{x^{d_\nu}}{N_\nu}} \right)^{r_\nu}  \prod\limits_{\lambda=1}^v \left( \dfrac{1- e^{-iT\log \frac{x^{e_\lambda}}{M_\lambda}}}{i \log \frac{x^{e_\lambda}}{M_\lambda}}\right)^{t_\lambda} d\xi,
			\]
			\[
			\alpha_n^{'} = \max \left(0, \dfrac{1}{2\pi} \log \left( n -\dfrac{1}{2}\right), \ldots, \dfrac{1}{2\pi d_\nu} \log\left( N_\nu- \dfrac{1}{2}\right), \ldots,   \dfrac{1}{2\pi e_\lambda}\log \left( M_\lambda- \dfrac{1}{2}\right), \ldots  \right),
			\]
			\[
			\beta_n^{'} = \min \left(U, \dfrac{1}{2\pi} \log \left( n +\dfrac{1}{2}\right), \ldots, \dfrac{1}{2\pi d_\nu} \log\left( N_\nu+ \dfrac{1}{2}\right), \ldots,   \dfrac{1}{2\pi e_\lambda} \log\left( M_\lambda + \dfrac{1}{2}\right), \ldots  \right).
			\]
			Making the substitution $\xi = \dfrac{u + \log n}{2\pi}$ in the integral $j(n)$, we obtain
			\[
			j(n) = \dfrac{1}{2\pi}  \int \limits_{\alpha_n}^{\beta_n} \hat{h}\left(\dfrac{u+\log n}{2\pi} \right) \prod\limits_{\nu=1}^s \left(  \dfrac{e^{iT(d_\nu u + \log \frac{n^{d_\nu}}{N_\nu})}-1}{i(d_\nu u + \log \frac{n^{d_\nu}}{N_\nu})}\right)^{r_\nu}  \prod\limits_{\lambda=1}^v \left(\dfrac{1-e^{-iT(e_\lambda u + \log \frac{n^{e_\lambda}}{M_\lambda})}}{i (e_\lambda u + \log \frac{n^{e_\lambda}}{M_\lambda})} \right)^{t_\lambda}du, 
			\]
			where $\alpha_n = 2\pi\alpha_n'-\log n$ and $\beta_n = 2\pi\beta_n'-\log n.$
			
			We now show that the contribution of those $N_\nu$ and $M_\lambda$ for which $n^{d_\nu} \ne N_\nu$ and $n^{e_\lambda} \ne M_\lambda$ is negligible. Without loss of generality, let $n^{d_1} \ne N_1$. Clearly, this is possible only if $d_1 \geqslant 2$. Indeed, for $d_1 = 1$ we have
			\[
			N_1 - \dfrac{1}{2} < x^{d_1} < N_1 + \dfrac{1}{2},
			\]
			\[
			n - \dfrac{1}{2} < x  < n + \dfrac{1}{2}.
			\]
			Hence $N_1 = n = n^{d_1}$, and we obtain a contradiction.
			
			We split the integral $j(n)$ into two parts:
			$$
			j(n) = j_n^{(1)} + j_{n,2},
			$$
			where $j_n^{(1)}$ includes exactly those $u$ from the integral $j(n)$ for which
			\begin{equation}\label{3stars}
				\left| d_1 u + \log \dfrac{n^{d_1}}{N_1}\right| \leqslant \dfrac{1}{\sqrt{T}}.
			\end{equation}
			We estimate the integral $j_n^{(1)}$. From \eqref{3stars} we obtain
			\begin{equation}\label{5stars}
				u = \dfrac{1}{d_1} \log \dfrac{N_1}{n^{d_1}} + \dfrac{\theta}{d_1\sqrt{T}} \ \ (|\theta| \leqslant 1).
			\end{equation}
			From this, since $N_1 \ne n^{d_1}$ and $n^{d_1}	 = o (\sqrt{T})$, we obtain
			\begin{multline*}
				|u| \geqslant \left| \dfrac{1}{d_1} \log \dfrac{N_1}{n^{d_1}}\right|  - \dfrac{1}{d_1\sqrt{T}} \geqslant\\
				\dfrac{1}{d_1} \left|\log \left(1 + \dfrac{1}{\min(n^{d_1}, N_1)}  \right)  \right| - \dfrac{1}{d_1\sqrt{T}} \gg \dfrac{1}{n^{d_1}} \gg \dfrac{1}{n^{O_{\textbf{\textit{a}}}(1)}}.
			\end{multline*}
			Using the inequalities
			\[
			\left|\dfrac{e^{iTw}-1}{iw} \right| \ll \min \left(T, \dfrac{1}{|w|} \right)
			\]
			and $|\hat{h}(\xi)| \ll 1$, we obtain
			\[
			j_n^{(1)} \ll \int\limits_{|\xi_1(u)| \leqslant {1}/{\sqrt{T}}} T^{r_1} \prod \limits_{\nu=2}^s \left( \min \left(T, \dfrac{1}{|\xi_\nu(u)|} \right)  \right)^{r_\nu}  \prod\limits_{\lambda=1}^v \left( \min \left(T, \dfrac{1}{|\eta_\lambda(u)|} \right)  \right)^{t_\lambda} du,
			\]	
			where 
			\[
			\xi_\nu(u) = d_\nu u + \log \dfrac{n^{d_\nu}}{N_\nu} \ (1 \leqslant \nu \leqslant s),
			\]	
			\[
			\eta_\lambda(u) = e_\lambda  u + \log \dfrac{n^{e_\lambda}}{M_\lambda} \ (1 \leqslant \lambda \leqslant v).
			\]
			
			Note that for any $2 \leqslant \nu \leqslant s$ we have $N_1^{d_\nu} \ne N_\nu^{d_1}$. Indeed, if $N_1^{d_\nu} = N_\nu^{d_1}$ for some $\nu \geqslant 2$, then, since $(d_\nu, d_1) = 1$, it follows that $N_1 = k^{d_1}$ for some $k \geqslant 1$. From this, in view of the inequality
			$$
			-\dfrac{1}{2} + \left( n - \dfrac{1}{2} \right)^{d_1} < N_1 < \left( n + \dfrac{1}{2} \right)^{d_1} + \dfrac{1}{2},
			$$
			which is impossible for $k \leqslant n - 1$ and $k \geqslant n + 1$, we conclude that $k = n$ and $N_1 = n^{d_1}$, which contradicts the assumption. Similarly, for any $1 \leqslant \lambda \leqslant v$ we have $N_1^{e_\lambda} \ne M_\lambda^{d_1}$.
			
			Now take an arbitrary $2 \leqslant \nu \leqslant s$ and consider the expression
			\[
			\xi_\nu (u) = d_\nu u + \log \dfrac{n^{d_\nu}}{N_\nu}.
			\]
			By virtue of formula \eqref{5stars} we have
			\[
			\xi_\nu(u) = \dfrac{d_\nu}{d_1} \log \dfrac{N_1}{n^{d_1}} + \log \dfrac{n^{d_\nu}}{N_\nu} + \dfrac{d_\nu}{d_1} \dfrac{\theta}{\sqrt{T}}.
			\]
			Since $N_1^{d_\nu} \ne N_\nu^{d_1}$, we have either $N_1^{d_\nu} \geqslant N_\nu^{d_1} + 1$ or $N_\nu^{d_1} \geqslant N_1^{d_\nu} + 1$.
			
			Suppose the first inequality holds. Then
			\[
			N_1^{\frac{1}{d_1}} \geqslant N_\nu^{\frac{1}{d_\nu}}\left( 1 + \dfrac{1}{N_\nu^{d_1}}\right) ^{\frac{1}{d_1d_\nu}}.
			\]
			Hence,
			\[
			\log \dfrac{N_1^{\frac{1}{d_1}} }{n} \geqslant \log \dfrac{N_\nu^{\frac{1}{d_\nu}} }{n} + \dfrac{1}{d_1d_\nu} \log \left(1 + \dfrac{1}{N_\nu^{d_1}} \right), 
			\]	
			\[
			\dfrac{1}{d_1} \log \dfrac{N_1}{n^{d_1}} \geqslant \dfrac{1}{d_\nu} \log \dfrac{N_\nu}{n^{d_\nu}} + \dfrac{1}{d_1 d_\nu} \log \left(1 + \dfrac{1}{N_\nu^{d_1}} \right) 
			\]
			and
			\[
			\dfrac{d_\nu}{d_1} \log \dfrac{N_1}{n^{d_1}} + \log \dfrac{n^{d_\nu}}{N_\nu} \geqslant \dfrac{1}{d_1} \log \left(1 + \dfrac{1}{N_\nu^{d_1}} \right) > 0. 
			\]	
			Therefore,
			\[
			|\xi_\nu(u)| \geqslant \dfrac{1}{d_1} \log \left(1 + \dfrac{1}{N_\nu^{d_1}} \right) - \dfrac{d_\nu}{d_1} \dfrac{1}{\sqrt{T}} \gg \dfrac{1}{N_\nu^{d_1}} \gg \dfrac{1}{n^{O_{\textbf{\textit{a}}}(1)}}.
			\]
			If the second inequality holds, then
			\[
			\log \dfrac{N_\nu}{n^{d_\nu}} - \dfrac{d_\nu}{d_1} \log \dfrac{N_1}{n^{d_1}} \geqslant  \dfrac{1}{d_1} \log \left(1 + \dfrac{1}{N_\nu^{d_1}} \right) > 0
			\]
			and 
			\[
			|\xi_\nu(u) | \geqslant  \dfrac{1}{d_1} \log \left(1 +\dfrac{1}{N_\nu^{d_1}} \right) - \dfrac{d_\nu}{d_1\sqrt{T}} \gg \dfrac{1}{n^{O_{\textbf{\textit{a}}}(1)}}.
			\]	
			Similarly, we obtain the inequalities
			\[
			|\eta_\lambda (u)| \gg \dfrac{1}{n^{O_{\textbf{\textit{a}}}(1)}} \ (1 \leqslant \lambda \leqslant v).
			\]
			Thus, since $r_1 \leqslant m - 1$, we obtain
			\[
			j_n^{(1)} \ll \int \limits_{|\xi_1(u)| \leqslant \frac{1}{\sqrt{T}}} T^{r_1}  n^{O_{\textbf{\textit{a}}}(1)} du \ll T^{r_1-\frac{1}{2}} n^{O_{\textbf{\textit{a}}}(1)} \ll T^{m-\frac{3}{2}} n^{O_{\textbf{\textit{a}}}(1)}.
			\]
			Since the number of possible choices for $N_1, N_2, \ldots, M_\nu$ does not exceed $n^{O_{\textbf{\textit{a}}}(1)}$, then, taking into account the trivial estimate $\Lambda(n) \leqslant \log n$, the contribution of the integral $j_n^{(1)}$ to $J$ does not exceed
			\[
			T^{m-\frac{3}{2}} \sum\limits_{n \leqslant e^{2\pi U} + \frac{1}{2}}  n^{O_{\textbf{\textit{a}}}(1)} \ll_{\varepsilon, \textbf{\textit{a}}} T^{m-\frac{3}{2} + \varepsilon}
			\]
			for sufficiently small $\varepsilon > 0$.
			
			We repeat the above procedure for the integral $j_{n,2}$. Namely, for the next $2 \leqslant \nu \leqslant s$ such that $N_\nu \ne n^{d_\nu}$, we split $j_{n,2}$ into two integrals: in the first, the inequality $|\xi_\nu(u)| \leqslant {1}/{\sqrt{T}}$ holds, and in the second, the opposite inequality holds. In this way, after $g = O_{\textbf{\textit{a}}}(1)$ steps, we obtain
			\[
			j(n) = j_n^{(1)} + j_n^{(2)} + \ldots + j_n^{(g)} + j_{n, g+1},
			\]
			where $j_n^{(1)}, j_n^{(2)}, \ldots, j_n^{(g)}$ contribute to $J$ no more than $O_{\textbf{\textit{a}}, \varepsilon} \left(T^{m-\frac{3}{2} +\varepsilon}  \right) $, and
			\[
			j_{n, g+1} \ll \int\limits_{\mathcal{G}} \prod \limits_{\nu=1}^s \left|\dfrac{e^{iT\xi_{\nu} (u)}-1}{i\xi_\nu(u)} \right|^{r_\nu}  \prod \limits_{\lambda=1}^v\left| \dfrac{1-e^{-iT\eta_\lambda(u)}}{i\eta_\lambda(u)} \right|^{t_\lambda} du,
			\]
			where the domain $\mathcal{G}$ is defined by the inequalities
			\begin{itemize}
				\item $\log \left(1 -\dfrac{1}{2n} \right)  \leqslant u \leqslant \log \left(1+\dfrac{1}{2n} \right)$;
				
				\item $|\xi_\nu(u)| > \dfrac{1}{\sqrt{T}}$ for each $1 \leqslant \nu \leqslant s$ such that $N_\nu \ne n^{d_\nu}$;
				
				\item $|\eta_\lambda(u)| > \dfrac{1}{\sqrt{T}}$ for each $1 \leqslant \lambda \leqslant v$ such that $M_\lambda \ne n^{e_\lambda}$.
			\end{itemize}
			Let
			\[
			P = \sum\limits_{\substack{1 \leqslant \nu \leqslant s \\ N_\nu \ne n^{d_\nu}}} r_\nu + \sum\limits_{\substack{1 \leqslant \lambda \leqslant v \\ M_\lambda \ne n^{e_\lambda}}} t_\lambda.
			\]
			Then $P \geqslant 1$ and 
			\begin{multline*}
				j_{n, g+1} \ll (\sqrt{T})^P \int\limits_{\mathcal{G}} \prod \limits_{\substack{1 \leqslant \nu \leqslant s \\ N_\nu = n^{d_\nu}}} \left| \dfrac{e^{iTd_\nu u}-1}{i d_\nu u} \right|^{r_\nu}  \prod\limits_{\substack{1 \leqslant \lambda \leqslant v \\ M_\lambda = n^{e_\lambda}}} \left| \dfrac{1- e^{-i Te_\lambda u}}{i e_\lambda u} \right|^{t_\lambda}du \ll \\
				T^{\frac{P}{2}} \int\limits_{\log\left(1-\frac{1}{2n} \right)}^{\log\left(1+\frac{1}{2n} \right)}\prod \limits_{\substack{1 \leqslant \nu \leqslant s \\ N_\nu = n^{d_\nu}}}  \left.\min\right.^{r_\nu} \left(T, \dfrac{1}{|d_\nu u|} \right)  \prod\limits_{\substack{1 \leqslant \lambda \leqslant v \\ M_\lambda = n^{e_\lambda}}} \left.\min\right.^{t_\lambda} \left(T, \dfrac{1}{|e_\lambda u|} \right) du \ll \\
				T^{\frac{P}{2}}  \int\limits_0^{+\infty} \left.\min\right.^{m-P} \left(T, \dfrac{1}{u} \right) du \ll T^{\frac{P}{2}} \left( T^{m-P} \int\limits_0^{\frac{1}{T}} du +\int\limits_{\frac{1}{T}}^{+\infty} \dfrac{du}{u^{m-P}} \right) \ll\\
				T^{m-\frac{P}{2}-1} \ll T^{m-\frac{3}{2}}.
			\end{multline*} 
			Hence the contribution of $j_{n, g+1}$ to $J$ will also not exceed  
			$O_{\mathbf{\textit{a}}, \varepsilon} \left( T^{\,m-\frac{3}{2}+\varepsilon} \right)$.  
			Thus,
			\begin{multline*}
				J = \dfrac{1}{2\pi} \sum\limits_{n \leqslant e^{2\pi U} + \frac{1}{2}} \prod \limits_{\nu=1}^s \left(- \dfrac{\Lambda(n^{d_\nu})}{2\pi \sqrt{n^{d_\nu}}} \right)^{r_\nu}  \prod \limits_{\lambda=1}^v \left( -\dfrac{\Lambda(n^{e_\lambda})}{2\pi \sqrt{n^{e_\lambda}}} \right)^{t_\lambda}  j(n) + O_{\textbf{\textit{a}}, \varepsilon} \left( T^{m-\frac{3}{2}+\varepsilon} \right) = \\
				\dfrac{(-1)^m}{(2\pi)^{m+1}} \sum\limits_{n \leqslant e^{2\pi U} + \frac{1}{2}} \dfrac{\Lambda^m(n)}{n^S} j(n) + O_{\textbf{\textit{a}}, \varepsilon} \left( T^{m-\frac{3}{2}+\varepsilon} \right),
			\end{multline*} 
			where \begin{multline*}
				S = \dfrac{r_1d_1 + \cdots + r_sd_s + e_1t_1 + \cdots + e_vt_v}{2} =\\
				\dfrac{b_1 + \cdots + b_q + c_1 + \cdots + c_\ell}{2} = b_1 + \cdots + b_q
			\end{multline*}
			and
			\[
			j(n) = \int\limits_{\alpha_n}^{\beta_n} \hat{h}\left( \dfrac{u +  \log n}{2\pi}\right) \prod\limits_{\nu=1}^s \left(\dfrac{e^{iTd_\nu u}-1}{i d_\nu u} \right)^{r_\nu} \prod \limits_{\lambda=1}^v \left(\dfrac{1-e^{-iTe_\lambda u}}{ie_\lambda u} \right)^{t_\lambda} du.
			\]
			Replace\ $\hat{h}\!\left(\dfrac{u+\log n}{2\pi}\right)$\ with\ $\hat{h}\!\left(\dfrac{\log n}{2\pi}\right)$.
			Since \[
			\hat{h}\left( \dfrac{u +  \log n}{2\pi}\right) - \hat{h}\left( \dfrac{ \log n}{2\pi}\right) = \int\limits_{\frac{\log n}{2\pi}}^{\frac{u+\log n}{2\pi}} \hat{h'}(\xi) d\xi \ll |u|,
			\]
			the error introduced by this substitution is at most
			\begin{multline*}
				\int\limits_{\log\left(1-\frac{1}{2n} \right) }^{\log\left(1+\frac{1}{2n} \right) } |u|  \prod\limits_{\nu=1}^s \left.\min\right.^{r_\nu} \left(T, \dfrac{1}{d_\nu |u|} \right)  \prod \limits_{\lambda=1}^v  \left.\min\right.^{t_\lambda} \left(T, \dfrac{1}{e_\lambda |u|} \right) \ll \\
				\int\limits_{-\infty}^{+\infty} |u| \left.\min\right.^{m}	\left(T, \dfrac{1}{|u|} \right) du \ll T^m \int\limits_0^{\frac{1}{T}} u du + \int\limits_{\frac{1}{T}}^{+\infty} \dfrac{du}{u^{m-1}} \ll T^{m-2}.
			\end{multline*} 	
			Hence,
			\begin{equation}\label{6stars}
				J = \dfrac{(-1)^m}{(2\pi)^{m+1}}\sum\limits_{n \leqslant e^{2\pi U} + \frac{1}{2}}  \dfrac{\Lambda^m(n)}{n^S} \hat{h} \left( \dfrac{\log n}{2\pi} \right) j'(n) + O_{\textbf{\textit{a}}, \varepsilon}\left( T^{m-\frac{3}{2}+\varepsilon} \right),
			\end{equation}
			where
			\[
			j'(n) = \int\limits_{\alpha_n}^{\beta_n} \prod\limits_{\nu=1}^s \left(  \dfrac{e^{iTd_\nu u}-1}{i d_\nu u}\right)^{r_\nu} \prod \limits_{\lambda=1}^v \left( \dfrac{1-e^{-iTe_\lambda u}}{i e_\lambda u}\right)^{t_\lambda} du.
			\]
			Extend the region of integration to the entire real line. Since
			$$
			\int\limits_{-\infty}^{\alpha_n}\frac{du}{u^m}+\int\limits_{\beta_n}^{\infty}\frac{du}{u^m}\ll n^{O_{\mathbf{\it a}}(1)},
			$$
			the error introduced by this extension is absorbed into the $O$-term in \eqref{6stars}.
			
			Since
			\begin{multline*}
				\int\limits_{-\infty}^{+\infty} \prod\limits_{\nu=1}^s  \left(  \dfrac{e^{iTd_\nu u}-1}{i d_\nu u}\right)^{r_\nu} \prod \limits_{\lambda=1}^v \left( \dfrac{1-e^{-iTe_\lambda u}}{i e_\lambda u}\right)^{t_\lambda} du = \\
				%\int\limits_{-\infty}^{+\infty}  \prod\limits_{\nu=1}^s e^{\frac{iT d_\nu u r_\nu}{2}} \prod \limits_{\lambda=1}^v  e^{-\frac{iT e_\lambda u t_\lambda}{2}} \cdot 2^{r_1 + \ldots + r_s + t_1 + \ldots + t_v} \cdot \prod\limits_{\nu=1}^s \left(\dfrac{e^{\frac{iTd_\nu u} {2}} -e^{\frac{iTd_\nu u} {2}}}{2i d_\nu u} \right)^{r_\nu}  \cdot \prod\limits_{\lambda=1}^v \left(\dfrac{e^{\frac{iTe_\lambda u} {2}} -e^{\frac{iTe_\lambda u} {2}}}{2i e_\lambda u} \right)^{e_\lambda} du = \\
				2^m \int\limits_{-\infty}^{+\infty} \prod\limits_{\nu=1}^s \left(\dfrac{\sin \left(\frac{Td_\nu u}{2}\right) }{d_\nu u}\right)^{r_\nu} \prod\limits_{\lambda=1}^v \left(\dfrac{\sin \left(\frac{Te_\lambda u}{2}\right) }{e_\lambda u}\right)^{e_\lambda} du = \\
				%\\ \dfrac{2^m}{\prod\limits_{k=1}^m |a_k|} \int\limits_{-\infty}^{+\infty} \prod\limits_{\nu=1}^s \left(\dfrac{\sin \left(\frac{Td_\nu u}{2}\right) }{  u}\right)^{r_\nu} \prod\limits_{\lambda=1}^v \left(\dfrac{\sin \left(\frac{Te_\lambda u}{2}\right) }{  u}\right)^{e_\lambda} du  = \\ 
				%\dfrac{2^m}{\prod\limits_{k=1}^m |a_k|} \cdot \left(\dfrac{T}{2} \right)^{m-1}  \int\limits_{-\infty}^{+\infty} \prod\limits_{\nu=1}^s \left(\dfrac{\sin \left(\frac{\frac{Tu}{2}d_\nu}{2}\right) }{ \frac{Tu}{2}}\right)^{r_\nu} \cdot \prod\limits_{\lambda=1}^v \left(\dfrac{\sin \left(\frac{\frac{Tu}{2} e_\lambda}{2}\right) }{\frac{Tu}{2}}\right)^{e_\lambda} d\left( \frac{Tu}{2}\right)  = \\
				{2T^{m-1}}\int\limits_{-\infty}^{+\infty} \prod\limits_{\nu=1}^s \left( \dfrac{\sin (w  d_\nu) }{w d_\nu}\right)^{r_\nu} \prod \limits_{\lambda=1}^v \left( \dfrac{\sin (w  e_\lambda) }{w e_\lambda}\right)^{t_\lambda} dw = \\
				{2T^{m-1}}\int\limits_{-\infty}^{+\infty} \prod\limits_{k=1}^m \dfrac{\sin (w |a_k|)}{w |a_k|} dw,
			\end{multline*}
			we get
			\[
			J = D(m)T^{m-1} \sum\limits_{n \leqslant e^{2\pi U} + \frac{1}{2}}\dfrac{\Lambda^m(n )}{n^S} \hat{h}\left( \dfrac{\log n}{2\pi}\right) + O_{\textbf{\textit{a}}, \varepsilon} \left( T^{m-\frac{3}{2}+\varepsilon}\right),
			\]
			where the value $D(m)$ is defined in the assumption.
			
			Extend the summation over $n$ to infinity. Since $m\geqslant 3$, we have $S\geqslant {3}/{2}$, and hence $S\geqslant 2$, it follows that
			\[
			\sum\limits_{n > e^{2\pi U} + \frac{1}{2}} \dfrac{\Lambda^m(n )}{n^S} \hat{h}\left( \dfrac{\log n}{2\pi}\right) \ll \sum\limits_{n \geqslant e^{2\pi U} + \frac{1}{2}} \dfrac{(\log n)^m}{n^{2}} \ll_m \dfrac{U^m}{e^{2\pi U}} \ll \dfrac{1}{\log T}.
			\]
			Therefore,
			\[
			J = D(m) T^{m-1} \sum\limits_{n=1}^{+\infty} \dfrac{\Lambda^m(n )}{n^S} \hat{h}\left( \dfrac{\log n}{2\pi}\right) + O_m \left( \dfrac{T^{m-1}}{\log T}\right).
			\]
			Using the equality
			\[
			\hat{h}\left( \dfrac{\log n}{2\pi}\right) = \int\limits_{-\infty}^{+\infty} h(t) n^{-it} dt
			\]
			and using the uniform convergence of the integral and the series, we obtain
			\[
			J = D(m) T^{m-1} \int\limits_{-\infty}^{+\infty} h(t) \sum\limits_{n=1}^{+\infty} \dfrac{\Lambda^m(n)}{n^{S+it}} dt + O_m \left( \dfrac{T^{m-1}}{\log T} \right).
			\]
			Taking into account all the estimates above, we obtain
			\begin{multline*}
				H = 2 \Re D(m) T^{m-1} \int\limits_{-\infty}^{+\infty} h(t) \sum\limits_{n=1}^{+\infty} \dfrac{\Lambda^m(n)}{n^{S+it}} dt + O_{m, \textbf{\textit{a}}} \left(\dfrac{T^{m-1}}{(\log T)^{\frac{1}{m}}} \right) =\\
				D(m) T^{m-1} \int\limits_{-\infty}^{+\infty} h(t) \left( \mathcal{K}_m(S+it) + \mathcal{K}_m(S-it)\right)  dt + O_{m, \textbf{\textit{a}}} \left(\dfrac{T^{m-1}}{(\log T)^{\frac{1}{m}}} \right),
			\end{multline*}
			where $$ \mathcal{K}_m(s) = \sum\limits_{n=1}^{+\infty} \dfrac{\Lambda^m(n)}{n^s}.$$
			Theorem \ref{Th1} is proved.

		\end{document}